\documentclass{amsart}
\usepackage{amsfonts,graphicx,rotating,slashbox,ctable}
\usepackage{amssymb}
\input diagxy

\newtheorem{theorem}{Theorem}[section]
\newtheorem{lemma}[theorem]{Lemma}
\newtheorem{corollary}[theorem]{Corollary}
\newtheorem{proposition}[theorem]{Proposition}

\theoremstyle{definition}
\newtheorem{definition}[theorem]{Definition}

\theoremstyle{remark}
\newtheorem{remark}[theorem]{Remark}

\numberwithin{equation}{section}

\makeatletter
\def\imod#1{\allowbreak\mkern5mu({\operator@font mod}\,\,#1)}
\makeatother

\newcommand{\DS}{\displaystyle}

\begin{document}

\title[Gr\"obner bases and immersion theorems for Grassmannians $G_{3,n}$]{Gr\"obner bases and some immersion theorems for Grassmann manifolds $G_{3,n}$}

\author{Zoran Z. Petrovi\'c}
\address{University of Belgrade,
  Faculty of mathematics,
  Studentski trg 16,
  Belgrade,
  Serbia}
\email{zoranp@matf.bg.ac.rs}
\thanks{The first author was partially supported by Ministry of Science and Environmental Protection of Republic of Serbia Project \#174032.}

\author{Branislav I. Prvulovi\'c}
\address{University of Belgrade,
  Faculty of mathematics,
  Studentski trg 16,
  Belgrade,
  Serbia}
\email{bane@matf.bg.ac.rs}
\thanks{The second author was partially supported by Ministry of Science and Environmental Protection of Republic of Serbia Project \#174034.}

\subjclass[2000]{Primary 57R42, 13P10, 55S45, 55R40}



\keywords{Grassmannians, immersions, Gr\"obner bases, modified Postnikov towers}

\begin{abstract}
A Gr\"obner basis for the ideal determining mod $2$ cohomology
of Grassmannian $G_{3,n}$ is obtained. This is used, along with the method of obstruction theory, to establish some new immersion results for these manifolds.
\end{abstract}

\maketitle



\section{Introduction}
\label{intro}

The theory of Gr\"obner bases is one of the most powerful tools for deciding whether a certain polynomial in two or more variables belongs to a given ideal. An example where this problem is of particular interest is the mod $2$ cohomology algebra of Grassmann manifold $G_{k,n}=O(n+k)/O(n)\times O(k)$. By Borel's description, this algebra is just the polynomial algebra on the Stiefel-Whitney classes $w_{1},w_{2},\dots,w_{k}$ of the canonical vector bundle $\gamma_{k}$ over $G_{k,n}$ modulo the ideal $I_{k,n}$ generated by the dual classes $\overline{w}_{n+1},\overline{w}_{n+2},\dots,\overline{w}_{n+k}$.

A reduced Gr\"obner basis for the ideal $I_{2,n}$ has been obtained in \cite{Petrovic}. Based on that result for odd $n$, some new immersions of
Grassmannians $G_{2,2l+1}$ were established.

In this paper, we construct a reduced Gr\"obner basis for the ideal
$I_{3,n}$ for all $n$. This result is stated in Theorem \ref{t3}. In Corollary \ref{c2} we present a convenient vector space basis for $H^{*}(G_{3,n};\mathbb{Z}_{2})$.

In Section \ref{immer} we consider the immersion dimension of Grassmanians $G_{3,n}$ (which is defined by $\mathrm{imm}(G_{3,n}):=\min \{d\mid G_{3,n}\mbox{ immerses into } \mathbb{R}^{d}\}$). Some lower bounds for $\mathrm{imm}(G_{3,n})$ were established by Oproiu in \cite{Oproiu} where he used the method of the Stiefel-Whitney classes. From the general result of Cohen (\cite{Cohen}), one has an upper bound for $\mathrm{imm}(G_{3,n})$ and it seems that there has been no improvement of this result up till now.

Using the Gr\"obner basis and modified Postnikov towers, we get the following new immersion results.

\begin{theorem}\label{theorem2} If $n\equiv 0\imod 4$, then $G_{3,n}$ immerses into $\mathbb{R}^{6n-3}$.
\end{theorem}
This theorem improves Cohen's result whenever $\alpha (3n)=2$ (where $\alpha (3n)$ denotes the number of ones in the binary expansion of $3n$). In particular, consider the case $n=2^{r}$, $r\geq 2$. By the result of Oproiu (\cite{Oproiu}), $\mathrm{imm}(G_{3,2^{r}})\geq 6\cdot 2^{r}-3$ and by Theorem \ref{theorem2}, $\mathrm{imm}(G_{3,2^{r}})\leq 6\cdot 2^{r}-3$, so
\[\mathrm{imm}(G_{3,2^{r}})=6\cdot 2^{r}-3.\]
Also, if $n=2^{r}+\DS\sum_{j=0}^{s}2^{r+1+2j}=2^{r}+2^{r+1}\cdot \frac{2^{2s+2}-1}{3}$ for some $r\geq 2$ and $s\geq 0$, we have that $3n=2^{r}+2^{r+2s+3}$, so $\alpha (3n)=2$.
Therefore, if $n$ is of this form, Theorem \ref{theorem2} decreases the upper bound for $\mathrm{imm}(G_{3,n})$ by one.

\begin{theorem}\label{theorem3} If $n\equiv 6\imod 8$, then $G_{3,n}$ immerses into $\mathbb{R}^{6n-5}$.
\end{theorem}
The best improvement of the general Cohen's result (\cite{Cohen}) obtained from Theorem \ref{theorem3} is in the case $n=2+\DS\sum_{j=1}^{s}2^{2j}$, $s\geq 1$. Then $3n=2+2^{2s+2}$ and so we are able to decrease the upper bound for $\mathrm{imm}(G_{3,n})$ by $3$. For example, by this theorem and Oproiu's result, we have that $29\leq \mathrm{imm}(G_{3,6})\leq 31$.

\begin{theorem}\label{thm1} If $n\geq 3$ and $n\equiv 1\imod 8$, then $G_{3,n}$ immerses into $\mathbb{R}^{6n-6}$.
\end{theorem}
This theorem improves Cohen's result whenever $\alpha (3n)<6$. For example, if $n=1+2^{r}+\DS\sum_{j=1}^{s}2^{r+2j-1}=1+2^{r}+2^{r+1}\cdot \frac{2^{2s}-1}{3}$ for some $r\geq 3$ and $s\geq 0$, we have that $3n=3+ 2^{r}+2^{r+2s+1}$, so $\alpha (3n)=4$. When $s=0$, i.e., $n=2^{r}+1$ ($r\geq 3$), by Theorem \ref{thm1} and Oproiu's result we have that $6\cdot 2^{r}-3\leq \mathrm{imm}(G_{3,2^{r}+1})\leq 6\cdot 2^{r}$.

\begin{theorem}\label{theorem4} If $n\geq 3$ and $n\equiv 2\imod 8$, then $G_{3,n}$ immerses into $\mathbb{R}^{6n-7}$.
\end{theorem}
Again, there are a number of cases in which Theorem \ref{theorem4} improves previously known results. In particular, when $n=2^{r}+2$, $r\geq 3$, we have an improvement by $3$. In this case, using Oproiu's result (\cite{Oproiu}) and this theorem, we have $6\cdot 2^{r}-3\leq \mathrm{imm}(G_{3,2^{r}+2})\leq 6\cdot 2^{r}+5$.

\medskip

In addition to these main results, in Theorem \ref{t10} we use Gr\"obner bases to give a
simple proof of some of Oproiu's results concerning lower bounds for $\mathrm{imm}(G_{3,n})$ (excluding the cases $n=2^{r}-2$ and $n=2^{r}-1$).

\section{Gr\"obner bases}
\label{groebner}

Throughout this section, we denote by $\mathbb{N}_{0}$ the set of all nonnegative integers and the set of all positive integers is denoted by $\mathbb{N}$.
\medskip

Let $G_{k,n}$ be the Grassmann manifold of unoriented $k$-dimensional vector
subspaces in $\mathbb{R}^{n+k}$. It is known that the cohomology algebra $H^{*}(G_{k,n};\mathbb{Z}_{2})$
is isomorphic to the quotient $\mathbb{Z}_{2}[w_{1},w_{2},\dots ,w_{k}]/I_{k,n}$ of the polynomial algebra $\mathbb{Z}_{2}[w_{1},w_{2},\dots ,w_{k}]$ by the ideal $I_{k,n}$ generated by polynomials
$\overline{w}_{n+1},\overline{w}_{n+2},\dots ,\overline{w}_{n+k}$. These are obtained from the equation
\[(1+w_{1}+w_{2}+\dots +w_{k})(1+\overline{w}_{1}+\overline{w}_{2}+\dots )=1,\]
that is
\begin{equation}\label{f2}
1+\overline{w}_{1}+\overline{w}_{2}+\dots =\frac{1}{1+w_{1}+w_{2}+\dots +w_{k}}=\DS\sum_{t\geq 0}(w_{1}+w_{2}+\dots +w_{k})^{t}
\end{equation}
\[
=\sum_{t\geq 0}\sum_{a_{1}+\dots +a_{k}=t}[a_{1},\dots ,a_{k}]w_{1}^{a_{1}}\cdots w_{k}^{a_{k}}=\sum_{a_{1},\dots ,a_{k}\geq 0}[a_{1},a_{2},\dots ,a_{k}]w_{1}^{a_{1}}w_{2}^{a_{2}}\cdots w_{k}^{a_{k}},
\]
where $[a_{1},a_{2},\dots ,a_{k}]$ ($a_{j}\in \mathbb{N}_{0}$) denotes the multinomial coefficient,
\[[a_{1},a_{2},\dots ,a_{k}]=\tfrac{(a_{1}+a_{2}+\dots +a_{k})!}{a_{1}! a_{2}! \cdots a_{k}!}=\bigl(\begin{smallmatrix} a_{1}+a_{2}+\dots +a_{k}\\ a_{1}\end{smallmatrix}\bigr)\dots\bigl(\begin{smallmatrix}a_{k-1}+a_{k}\\a_{k-1}\end{smallmatrix}\bigr).\]
By identifying the homogenous parts of (cohomological) degree $r$ in formula (\ref{f2}), we obtain the following proposition.


\begin{proposition}\label{p0} For $r\in \mathbb{N}$,
\[\overline{w}_{r}=\DS\sum_{a_{1}+2a_{2}+\dots +ka_{k}=r}[a_{1},a_{2},\dots ,a_{k}]w_{1}^{a_{1}}w_{2}^{a_{2}}\cdots w_{k}^{a_{k}}.\]
\end{proposition}
It is understood that $a_{1},a_{2},\dots,a_{k}\in \mathbb{N}_{0}$.

For $k=3$ (which is the case from now on), Proposition \ref{p0} gives us
\[\overline{w}_{r}=
\DS\sum_{a+2b+3c=r}\bigl(\begin{smallmatrix}a+b+c\\ a\end{smallmatrix}\bigr)\bigl(\begin{smallmatrix}b+c\\ b\end{smallmatrix}\bigr )w_{1}^{a}w_{2}^{b}w_{3}^{c}, \qquad r\in \mathbb{N}.\]

Let $\preceq$ be the grlex ordering on the monomials in $\mathbb{Z}_{2}[w_{1},w_{2},w_{3}]$ (with $w_{1}>w_{2}>w_{3}$). This means that $w_{1}^{a}w_{2}^{b}w_{3}^{c}\prec w_{1}^{d}w_{2}^{e}w_{3}^{f}$ if one of the following three conditions holds:
\begin{itemize}
\item[$\mathrm{(i)}$] $a+b+c<d+e+f$;
\item[$\mathrm{(ii)}$] $a+b+c=d+e+f$ and $a<d$;
\item[$\mathrm{(iii)}$] $a+b+c=d+e+f$, $a=d$ and $b<e$.
\end{itemize}
Of course, $w_{1}^{a}w_{2}^{b}w_{3}^{c}\preceq w_{1}^{d}w_{2}^{e}w_{3}^{f}$ will mean that either $w_{1}^{a}w_{2}^{b}w_{3}^{c}\prec w_{1}^{d}w_{2}^{e}w_{3}^{f}$ or $w_{1}^{a}w_{2}^{b}w_{3}^{c}= w_{1}^{d}w_{2}^{e}w_{3}^{f}$.

\medskip

Let $n\geq 3$ be a fixed integer. In order to find a Gr\"obner basis for the ideal $I_{3,n}=(\overline{w}_{n+1},\overline{w}_{n+2},\overline{w}_{n+3})$, we define the polynomials $g_{m,l}\in \mathbb{Z}_{2}[w_{1},w_{2},w_{3}]$.

\begin{definition}\label{d2}  For $m,l\in \mathbb{N}_{0}$, let 
\[g_{m,l}:=\sum_{a+2b+3c=n+1+m+2l}\bigl(\begin{smallmatrix}a+b+c-m-l\\ a\end{smallmatrix}\bigr) \bigl(\begin{smallmatrix}b+c-l\\b \end{smallmatrix}\bigr)w_{1}^{a}w_{2}^{b}w_{3}^{c}.\]
\end{definition}
As before, it is understood that $a,b,c\in \mathbb{N}_{0}$.

Let us remark first that $g_{0,0}=\overline{w}_{n+1}$.

Secondly, we note that the coefficient $\bigl(\begin{smallmatrix}a+b+c-m-l\\ a\end{smallmatrix}\bigr) \bigl(\begin{smallmatrix}b+c-l\\b \end{smallmatrix}\bigr)$ may be nonzero when $a+b+c-m-l<0$ (or $b+c-l<0$). For example, if $n=4$ we have 
\[g_{5,0}=\!\!\! \sum_{a+2b+3c=10}\!\!\!\bigl(\begin{smallmatrix}a+b+c-5\\ a\end{smallmatrix}\bigr)\bigl(\begin{smallmatrix}b+c\\ b\end{smallmatrix}\bigr)w_{1}^{a}w_{2}^{b}w_{3}^{c}
=\bigl(\begin{smallmatrix}0\\ 0\end{smallmatrix}\bigr)\bigl(\begin{smallmatrix}5\\ 5\end{smallmatrix}\bigr)w_{2}^{5}+\bigl(\begin{smallmatrix}-1\\ 1\end{smallmatrix}\bigr)\bigl(\begin{smallmatrix}3\\ 0\end{smallmatrix}\bigr) w_{1}w_{3}^{3}=w_{2}^{5}+w_{1}w_{3}^{3}.\]

However, we can prove the following lemma.

\begin{lemma}\label{l2}  Let $a,b,c,m,l$ be nonnegative integers. Then the following implication holds: 
\[\bigl(\begin{smallmatrix}a+b+c-m-l\\ a\end{smallmatrix}\bigr)\bigl(\begin{smallmatrix}b+c-l\\ b\end{smallmatrix}\bigr)\neq 0
\Longrightarrow a+b+c<m+l \quad \mbox{or} \quad (b+c\geq m+l \quad \mbox{and} \quad c\geq l).\]
\end{lemma}

\begin{proof} Assume that ${a+b+c-m-l\choose a}{b+c-l\choose b}\neq 0$ and $a+b+c\geq m+l$. Then we have that ${a+b+c-m-l\choose a}\neq 0$ and since both $a+b+c-m-l$ and $a$ are nonnegative we conclude that $a+b+c-m-l\geq a$, i.e., $b+c\geq m+l$.

If $c<l$, then $b+c-l<b$ and since ${b+c-l\choose b}\neq 0$ it must be $b+c-l<0$. From this we have $0\leq a+b+c-m-l<a-m\leq a$, but this implies that ${a+b+c-m-l\choose a}=0$ contradicting the assumption ${a+b+c-m-l\choose a}{b+c-l\choose b}\neq 0$. This contradiction proves that $c\geq l$.
\end{proof}

\medskip

Finally, we define the set $G\subseteq \mathbb{Z}_{2}[w_{1},w_{2},w_{3}]$, our candidate for the Gr\"obner basis.

\begin{definition}\label{d3}  $G:=\{ g_{m,l}\mid m+l\leq n+1, m,l\in \mathbb{N}_{0}\}$.
\end{definition}

\medskip

We now prove an important property of $G$.

\begin{proposition}\label{p2} For $m,l\in \mathbb{N}_{0}$ such that $m+l\leq n+1$, we have that the leading term $\mathrm{LT}(g_{m,l})=w_{1}^{n+1-m-l}w_{2}^{m}w_{3}^{l}$ and all other terms (monomials) appearing in $g_{m,l}$ have the sum of the exponents $<n+1$.
\end{proposition}

\begin{proof} Obviously, the (nonnegative) integers $a:=n+1-m-l$, $b:=m$, $c:=l$ satisfy the conditions
$a+2b+3c=n+1+m+2l$ and ${a+b+c-m-l\choose a}{b+c-l\choose b}={a\choose a}{b\choose b}=1$ and so the monomial $w_{1}^{n+1-m-l}w_{2}^{m}w_{3}^{l}$ does appear in $g_{m,l}$.

\medskip

Now, it suffices to prove the inequality $a+b+c<n+1$ for all other monomials $w_{1}^{a}w_{2}^{b}w_{3}^{c}$ appearing in $g_{m,l}$. If the monomial $w_{1}^{a}w_{2}^{b}w_{3}^{c}$ is a summand in $g_{m,l}$, then $a+2b+3c=n+1+m+2l$ (i.e., $a=n+1+m+2l-2b-3c$) and ${a+b+c-m-l\choose a}{b+c-l\choose b}\equiv 1\imod 2$. According to Lemma \ref{l2}, $a+b+c<m+l$ or $b+c\geq m+l$ and $c\geq l$.

\medskip

In the first case $a+b+c<m+l\leq n+1$ and we are done.

\medskip

Otherwise, $b+c\geq m+l$ and $c\geq l$ give us that $b+2c\geq m+2l$ where the equality holds only if $c=l$ and $b=m$. But then $a=n+1+m+2l-2b-3c=n+1-m-l$ and since $w_{1}^{a}w_{2}^{b}w_{3}^{c}\neq w_{1}^{n+1-m-l}w_{2}^{m}w_{3}^{l}$, we actually have $b+2c>m+2l$. This implies that $a+b+c=n+1+m+2l-b-2c<n+1$.
\end{proof}

In what follows, we use the well-known formula ${a\choose b}={a-1\choose b}+{a-1\choose b-1}$,
$a,b\in \mathbb{Z}$ and its mod $2$ equivalents ${a\choose b}+{a-1\choose b-1}\equiv {a-1\choose b}\imod 2$ and ${a-1\choose b-1}\equiv {a\choose b}+{a-1\choose b}\imod 2$, $a,b\in \mathbb{Z}$ (it is understood that ${a\choose b}=0$ if $b$ is negative).

\medskip

Let $I_{G}$ be the ideal in $\mathbb{Z}_{2}[w_{1},w_{2},w_{3}]$ generated by $G$. Eventually, we shall prove that $I_{G}=I_{3,n}=(\overline{w}_{n+1},\overline{w}_{n+2},\overline{w}_{n+3})$, but for the moment we prove that $I_{G}$ contains $I_{3,n}$.

\begin{proposition}\label{p3} $I_{3,n}\subseteq I_{G}$.
\end{proposition}

\begin{proof} As we have already noticed, $\overline{w}_{n+1}=g_{0,0}\in I_{G}$.

Since
\begin{eqnarray*}
 w_{1}g_{0,0}+g_{1,0} & = & 
w_{1}\mkern-36mu\sum_{a+2b+3c=n+1}\mkern-18mu\bigl(\begin{smallmatrix}a+b+c\\ a\end{smallmatrix}\bigr)\bigl(\begin{smallmatrix}b+c\\ b\end{smallmatrix}\bigr)w_{1}^{a}w_{2}^{b}w_{3}^{c}
\,\,+\mkern-36mu\sum_{a+2b+3c=n+2}\mkern-18mu\bigl(\begin{smallmatrix}a+b+c-1\\ a\end{smallmatrix}\bigr)\bigl(\begin{smallmatrix}b+c\\ b\end{smallmatrix}\bigr)w_{1}^{a}w_{2}^{b}w_{3}^{c}\\
& = &\mkern-36mu \sum_{a+2b+3c=n+1}\mkern-18mu\bigl(\begin{smallmatrix}a+b+c\\ a\end{smallmatrix}\bigr)\bigl(\begin{smallmatrix}b+c\\ b\end{smallmatrix}\bigr)w_{1}^{a+1}w_{2}^{b}w_{3}^{c}
\,\,+\mkern-36mu\sum_{a+2b+3c=n+2}\mkern-18mu\bigl(\begin{smallmatrix}a+b+c-1\\ a\end{smallmatrix}\bigr)\bigl(\begin{smallmatrix}b+c\\ b\end{smallmatrix}\bigr)w_{1}^{a}w_{2}^{b}w_{3}^{c}\\
& = & \mkern-36mu\sum_{a+2b+3c=n+2}\mkern-18mu\bigl(\begin{smallmatrix}a+b+c-1\\ a-1\end{smallmatrix}\bigr)\bigl(\begin{smallmatrix}b+c\\ b\end{smallmatrix}\bigr)w_{1}^{a}w_{2}^{b}w_{3}^{c}
\,\,+\mkern-36mu\sum_{a+2b+3c=n+2}\mkern-18mu\bigl(\begin{smallmatrix}a+b+c-1\\ a\end{smallmatrix}\bigr)\bigl(\begin{smallmatrix}b+c\\ b\end{smallmatrix}\bigr)w_{1}^{a}w_{2}^{b}w_{3}^{c}\\
& = & \mkern-36mu\sum_{a+2b+3c=n+2}\bigl(\begin{smallmatrix}a+b+c\\ a\end{smallmatrix}\bigr)\bigl(\begin{smallmatrix}b+c\\ b\end{smallmatrix}\bigr)w_{1}^{a}w_{2}^{b}w_{3}^{c}\\
& = & \overline{w}_{n+2},
\end{eqnarray*}
we conclude that $\overline{w}_{n+2}=w_{1}g_{0,0}+g_{1,0}\in I_{G}$. Let us remark that the change of variable $a\mapsto a-1$ was made in the first sum, but we can still assume that $a\geq 0$ since ${a+b+c-1\choose a-1}$ is obviously equal to zero for $a=0$.

In order to show that $\overline{w}_{n+3}\in I_{G}$ we calculate:
\begin{eqnarray*}
w_{1}^{2}g_{0,0}+g_{2,0}
& = & \mkern-36mu\sum_{a+2b+3c=n+1}\bigl(\begin{smallmatrix}a+b+c\\ a\end{smallmatrix}\bigr)\bigl(\begin{smallmatrix}b+c\\ b\end{smallmatrix}\bigr)w_{2}^{b}w_{3}^{c}
\,\,+\mkern-36mu\sum_{a+2b+3c=n+3}\bigl(\begin{smallmatrix}a+b+c-2\\ a\end{smallmatrix}\bigr)\bigl(\begin{smallmatrix}b+c\\ b\end{smallmatrix}\bigr)w_{1}^{a}w_{2}^{b}w_{3}^{c}\\
& = &\mkern-36mu \sum_{a+2b+3c=n+3}\mkern-18mu\bigl(\begin{smallmatrix}a+b+c-2\\ a-2\end{smallmatrix}\bigr)\bigl(\begin{smallmatrix}b+c\\ b\end{smallmatrix}\bigr)w_{1}^{a}w_{2}^{b}w_{3}^{c}
+\mkern-36mu\sum_{a+2b+3c=n+3}\mkern-18mu\bigl(\begin{smallmatrix}a+b+c-2\\ a\end{smallmatrix}\bigr)\bigl(\begin{smallmatrix}b+c\\ b\end{smallmatrix}\bigr)w_{1}^{a}w_{2}^{b}w_{3}^{c}.
\end{eqnarray*}

First, we note that the change of variable $a\mapsto a-2$ in the first sum does not affect the requirement that $a$ runs through $\mathbb{N}_{0}$ since for $a=0$ and $a=1$ the binomial coefficient ${a+b+c-2\choose a-2}$ is equal to zero. Also,
${a+b+c-2\choose a-2}+{a+b+c-2\choose a}\equiv {a+b+c-1\choose a-1}+{a+b+c-2\choose a-1}+{a+b+c-2\choose a}= {a+b+c-1\choose a-1}+{a+b+c-1\choose a}={a+b+c\choose a}\imod 2$, so we have
\[w_{1}^{2}g_{0,0}+g_{2,0}=\sum_{a+2b+3c=n+3}\bigl(\begin{smallmatrix}a+b+c\\ a\end{smallmatrix}\bigr)\bigl(\begin{smallmatrix}b+c\\ b\end{smallmatrix}\bigr)w_{1}^{a}w_{2}^{b}w_{3}^{c}=\overline{w}_{n+3}\]
and the proposition is proved.
\end{proof}

In the subsequent calculations, the polynomials $g_{m,l}$ with $m+l=n+2$ will take part. We note that these polynomials are not necessarily elements of $G$, but, as Proposition \ref{p4} below states, they can be written as sums of some elements of $G$ (with possibility that this sum is empty, i.e., $g_{m,l}=0$). 

In order to achieve this kind of presentation for $g_{m,l}$ ($m+l=n+2$), we prove the crucial fact which is stated in the following lemma. (We recall that the integer $n\geq 3$ is fixed.)

\begin{lemma}\label{l3}  Let $m,l,a,b,c$ be nonnegative integers such that $m+l=n+2$ and $a+2b+3c=n+1+m+2l$. Then the following congruence holds:
\[\sum_{j=0}^{[\frac{m}{2}]}\bigl(\begin{smallmatrix}m-j\\ j\end{smallmatrix}\bigr)\bigl(\begin{smallmatrix}a+b+c-n-2+j\\ a\end{smallmatrix}\bigr)\bigl(\begin{smallmatrix}b+c-l-j\\ b\end{smallmatrix}\bigr)\equiv 0\imod 2,\]
or, singling out the summand for $j=0$,
\[\bigl(\begin{smallmatrix}a+b+c-n-2\\ a\end{smallmatrix}\bigr)\bigl(\begin{smallmatrix}b+c-l\\ b\end{smallmatrix}\bigr)\equiv
 \sum_{j=1}^{[\frac{m}{2}]}\bigl(\begin{smallmatrix}m-j\\ j\end{smallmatrix}\bigr)\bigl(\begin{smallmatrix}a+b+c-n-2+j\\ a\end{smallmatrix}\bigr)\bigl(\begin{smallmatrix}b+c-l-j\\ b\end{smallmatrix}\bigr)\imod 2.\]
\end{lemma}

\begin{proof} We prove the lemma by induction on $m$. Let
\[S(m,l,a,b,c):=\sum_{j=0}^{[\frac{m}{2}]}\bigl(\begin{smallmatrix}m-j\\ j\end{smallmatrix}\bigr)\bigl(\begin{smallmatrix}a+b+c-n-2+j\\ a\end{smallmatrix}\bigr)\bigl(\begin{smallmatrix}b+c-l-j\\ b\end{smallmatrix}\bigr).\]

\medskip

The induction base will consist of three parts: $m=0$, $m=1$ and $m=2$.

\medskip

Take $m=0$ and nonnegative integers $l,a,b,c$ such that $l=n+2$ and $a+2b+3c=n+1+2l$.
The statement of the lemma in this case simplifies to:
\[S(0,l,a,b,c)=\bigl(\begin{smallmatrix}a+b+c-n-2\\ a\end{smallmatrix}\bigr)\bigl(\begin{smallmatrix}b+c-n-2\\ b\end{smallmatrix}\bigr)\equiv 0\imod 2.\]

Since $a+2b+3c=n+1+2l=3n+5$, we have that $3c\leq a+2b+3c=3n+5$, so $c\leq n+\frac{5}{3}<n+2$, i.e., $b+c-n-2<b$.

If $b+c-n-2\geq 0$, then ${b+c-n-2\choose b}=0$ and we are done.

If $b+c-n-2<0$, then $a+b+c-n-2<a$. Also, $3(a+b+c)\geq a+2b+3c=3n+5$ implying $a+b+c\geq n+\frac{5}{3}$. But since $a+b+c$ is an integer, we actually have that $a+b+c\geq n+2$. So, $0\leq a+b+c-n-2<a$, and we conclude that ${a+b+c-n-2\choose a}=0$.

Thus, we have proved that $S(0,l,a,b,c)$ is actually equal to $0$.

\medskip

For $m=1$, take $l:=n+1$ and $a,b,c\geq 0$ such that $a+2b+3c=n+1+1+2l=3n+4$. In this case we need to prove
\[S(1,l,a,b,c)=\bigl(\begin{smallmatrix}a+b+c-n-2\\ a\end{smallmatrix}\bigr)\bigl(\begin{smallmatrix}b+c-n-1\\ b\end{smallmatrix}\bigr)\equiv 0\imod 2.\]
As in the case $m=0$, we obtain that $a+b+c\geq n+2$ and $c\leq n+1$. If $c<n+1$, the proof is analogous to that of the first case. If $c=n+1$, then, since $a+2b+3c=3n+4$, $a$ must be $1$ and $b$ must be $0$ and we obtain ${a+b+c-n-2\choose a}{b+c-n-1\choose b}={0\choose 1}{0\choose 0}=0$.

Again, we have proved that $S(1,l,a,b,c)=0$.

\medskip

If $m=2$, then $l=n$ and let $a,b,c$ be nonnegative integers such that $a+2b+3c=n+1+2+2l=3n+3$. Now, $S(2,l,a,b,c)$ has two summands and the statement of the lemma in this case reduces to the mod $2$ congruence 
\[\bigl(\begin{smallmatrix}a+b+c-n-2\\ a\end{smallmatrix}\bigr)\bigl(\begin{smallmatrix}b+c-n\\ b\end{smallmatrix}\bigr)+\bigl(\begin{smallmatrix}a+b+c-n-1\\ a\end{smallmatrix}\bigr)\bigl(\begin{smallmatrix}b+c-n-1\\ b\end{smallmatrix}\bigr)\equiv 0.\]
From the condition $a+2b+3c=3n+3$ we can deduce that $a+b+c\geq n+1$ and $c\leq n+1$.

If $c=n+1$, then necessary $a=b=0$, and we have 
\[S(2,l,a,b,c)=S(2,l,0,0,n+1)=\bigl(\begin{smallmatrix}-1\\ 0\end{smallmatrix}\bigr)\bigl(\begin{smallmatrix}1\\ 0\end{smallmatrix}\bigr)+\bigl(\begin{smallmatrix}0\\ 0\end{smallmatrix}\bigr)\bigl(\begin{smallmatrix}0\\ 0\end{smallmatrix}\bigr)=1+1\equiv 0\imod 2.\]

If $a+b+c=n+1$, since $0\leq c\leq b+c\leq a+b+c$ and $c+(b+c)+(a+b+c)=3(n+1)$, we conclude that $c$ must be $n+1$ and
this case reduces to the previous one.

Suppose now that $a+b+c\geq n+2$ and $c\leq n$. If $c<n$, then by the method of the case $m=0$ one proves that both summands must be zero. If $c=n$, then there are two possibilities for the pair $(a,b)$ such that the condition $a+2b+3c=3n+3$ is satisfied. First, if $a=3$ and $b=0$, we have
\[S(2,l,a,b,c)=\bigl(\begin{smallmatrix}1\\ 3\end{smallmatrix}\bigr)\bigl(\begin{smallmatrix}0\\ 0\end{smallmatrix}\bigr)+\bigl(\begin{smallmatrix}2\\ 3\end{smallmatrix}\bigr)\bigl(\begin{smallmatrix} -1\\ 0\end{smallmatrix}\bigr) =0+0=0.\]
Finally, if $a=b=1$, we obtain \[S(2,l,a,b,c)=\bigl(\begin{smallmatrix} 0\\ 1\end{smallmatrix}\bigr) \bigl(\begin{smallmatrix} 1\\ 1\end{smallmatrix}\bigr) +\bigl(\begin{smallmatrix} 1\\ 1\end{smallmatrix}\bigr) \bigl(\begin{smallmatrix} 0\\ 1\end{smallmatrix}\bigr)
=0+0=0,\]
and the basis for the induction is completed.

\medskip

For the induction step take $m\geq 3$, nonnegative integers $l,a,b,c$ such that $m+l=n+2$ and $a+2b+3c=n+1+m+2l$ and suppose that the statement of the lemma is true for all nonnegative integers $<m$. We need to prove that $S(m,l,a,b,c)$ is an even integer. Since ${m-j\choose j}={m-1-j\choose j}+{m-1-j\choose j-1}$, we have:
\[S(m,l,a,b,c)=\]
\[=\underbrace{\sum_{j=0}^{[\frac{m}{2}]}\bigl(\begin{smallmatrix} m-1-j\\ j \end{smallmatrix}\bigr) \bigl(\begin{smallmatrix} a+b+c-n-2+j\\ a\end{smallmatrix}\bigr) \bigl(\begin{smallmatrix} b+c-l-j\\ b \end{smallmatrix}\bigr)}_{S_1}
+\underbrace{\sum_{j=0}^{[\frac{m}{2}]}\bigl(\begin{smallmatrix} m-1-j\\ j-1\end{smallmatrix}\bigr) \bigl(\begin{smallmatrix} a+b+c-n-2+j\\ a\end{smallmatrix}\bigr) \bigl(\begin{smallmatrix} b+c-l-j\\ b \end{smallmatrix}\bigr)}_{S_2} .\]
Since ${b+c-l-j\choose b}={b+c-l-j-1\choose b}+{b+c-l-j-1\choose b-1}$, we obtain that $S_1$ is equal to:
\[\underbrace{\sum_{j=0}^{[\frac{m}{2}]}\bigl(\begin{smallmatrix} m-1-j\\ j\end{smallmatrix}\bigr) \bigl(\begin{smallmatrix} a+b+c-n-2+j\\ a\end{smallmatrix}\bigr) \bigl(\begin{smallmatrix} b+c-l-j-1\\ b\end{smallmatrix}\bigr)}_{S_3}
+\underbrace{\sum_{j=0}^{[\frac{m}{2}]}\bigl(\begin{smallmatrix} m-1-j\\ j \end{smallmatrix}\bigr) \bigl(\begin{smallmatrix} a+b+c-n-2+j\\ a\end{smallmatrix}\bigr) \bigl(\begin{smallmatrix} b+c-l-j-1\\ b-1\end{smallmatrix}\bigr)}_{S_4} .\]
So, $S(m,l,a,b,c)=S_{2}+S_{3}+S_{4}$.

First, we consider the sum $S_{4}$. If $m$ is odd, then $[\frac{m}{2}]=[\frac{m-1}{2}]$ and if $m$ is even, say $m=2r$ ($r\geq 2$), then the first factor of the last summand in the sum $S_{4}$ (for $j=[\frac{m}{2}]=r$) is ${r-1\choose r}=0$, so in either case
\begin{eqnarray*}
S_{4} & = & \sum_{j=0}^{[\frac{m-1}{2}]}\bigl(\begin{smallmatrix} m-1-j\\ j\end{smallmatrix}\bigr) \bigl(\begin{smallmatrix} a+b+c-n-2+j\\ a\end{smallmatrix}\bigr) \bigl(\begin{smallmatrix} b+c-l-j-1\\ b-1\end{smallmatrix}\bigr) \\
& = & S(m-1,l+1,a,b-1,c+1)\equiv 0\imod 2,
\end{eqnarray*}

\medskip
\noindent by the induction hypothesis if $b>0$ and if $b=0$ it is obvious that $S_{4}=0$.

Now, we have $S(m,l,a,b,c)\equiv S_{2}+S_{3}\imod 2$ and we consider the sum $S_{3}$. Since ${m-1-j\choose j}={m-2-j\choose j}+{m-2-j\choose j-1}$, $S_3$ can be written as the sum:
\[\underbrace{\sum_{j=0}^{[\frac{m}{2}]}\bigl(\begin{smallmatrix} m-2-j\\ j\end{smallmatrix}\bigr) \bigl(\begin{smallmatrix} a+b+c-n-2+j\\ a\end{smallmatrix}\bigr) \bigl(\begin{smallmatrix} b+c-l-j-1\\ b\end{smallmatrix}\bigr)}_{S_5}
+\underbrace{\sum_{j=0}^{[\frac{m}{2}]}\bigl(\begin{smallmatrix} m-2-j\\ j-1\end{smallmatrix}\bigr) \bigl(\begin{smallmatrix} a+b+c-n-2+j\\ a\end{smallmatrix}\bigr) \bigl(\begin{smallmatrix} b+c-l-j-1\\ b \end{smallmatrix}\bigr)}_{S_6} .\]
So, we have the congruence $S(m,l,a,b,c)\equiv S_{2}+S_{5}+S_{6}\imod 2$.

Consider the sum $S_{5}$ and its summand for $j=[\frac{m}{2}]$. The first factor of this summand is ${m-2-[\frac{m}{2}]\choose [\frac{m}{2}]}$. If $m=3$, this binomial coefficient equals ${0\choose 1}=0$. If $m\geq 4$, we have that $m-2-[\frac{m}{2}]\geq [\frac{m}{2}]-2\geq 0$. Also, $\frac{m}{2}-1<[\frac{m}{2}]$ implying $m-2-[\frac{m}{2}]<[\frac{m}{2}]$. We conclude that ${m-2-[\frac{m}{2}]\choose [\frac{m}{2}]}=0$, i.e., the summand obtained for $j=[\frac{m}{2}]$ is zero and so: 
\begin{eqnarray*}
S_5 & = & \sum_{j=0}^{[\frac{m}{2}]-1}\bigl(\begin{smallmatrix} m-2-j\\ j\end{smallmatrix}\bigr) \bigl(\begin{smallmatrix} a+b+c-n-2+j\\ a\end{smallmatrix}\bigr) \bigl(\begin{smallmatrix} b+c-l-j-1\\ b\end{smallmatrix}\bigr) \\
& = &\sum_{j=0}^{[\frac{m-2}{2}]}\bigl(\begin{smallmatrix} m-2-j\\ j\end{smallmatrix}\bigr) \bigl(\begin{smallmatrix} a+b+c-n-2+j\\ a\end{smallmatrix}\bigr) \bigl(\begin{smallmatrix} b+c-l-j-1\\ b\end{smallmatrix}\bigr) .
\end{eqnarray*}

By looking at the sum $S_{2}$ one easily sees that the first summand (for $j=0$) equals zero (since ${m-1\choose -1}=0$). This means that
\begin{eqnarray*}
S_{2} & = &\sum_{j=1}^{[\frac{m}{2}]}\bigl(\begin{smallmatrix} m-1-j\\ j-1\end{smallmatrix}\bigr) \bigl(\begin{smallmatrix} a+b+c-n-2+j\\ a\end{smallmatrix}\bigr) \bigl(\begin{smallmatrix} b+c-l-j\\ b\end{smallmatrix}\bigr) \\
& = & \sum_{j=0}^{[\frac{m}{2}]-1}\bigl(\begin{smallmatrix} m-1-j-1\\ j\end{smallmatrix}\bigr) \bigl(\begin{smallmatrix} a+b+c-n-2+j+1\\ a\end{smallmatrix}\bigr) \bigl(\begin{smallmatrix} b+c-l-j-1\\ b\end{smallmatrix}\bigr) \\
& = & \sum_{j=0}^{[\frac{m-2}{2}]}\bigl(\begin{smallmatrix} m-2-j\\ j\end{smallmatrix}\bigr) \bigl(\begin{smallmatrix} a+b+c-n-1+j\\ a\end{smallmatrix}\bigr) \bigl(\begin{smallmatrix} b+c-l-j-1\\ b\end{smallmatrix}\bigr) .
\end{eqnarray*}

Now the sums $S_{2}$ and $S_{5}$ are similar and since ${a+b+c-n-1+j\choose a}+{a+b+c-n-2+j\choose a}\equiv {a+b+c-n-2+j\choose a-1}\imod 2$, we have that
\begin{eqnarray*}
S_{2}+S_{5} & \equiv & \sum_{j=0}^{[\frac{m-2}{2}]}\bigl(\begin{smallmatrix} m-2-j\\ j\end{smallmatrix}\bigr) \bigl(\begin{smallmatrix} a+b+c-n-2+j\\ a-1\end{smallmatrix}\bigr) \bigl(\begin{smallmatrix} b+c-l-j-1\\ b\end{smallmatrix}\bigr) \\
& = & S(m-2,l+2,a-1,b,c+1)\equiv 0\imod 2.
\end{eqnarray*}

Again, we note that the upper sum is zero if $a=0$ and if $a>0$ we apply the induction hypothesis and obtain the latter congruence.

We have reached the congruence $S(m,l,a,b,c)\equiv S_{6}\imod 2$. Finally, by considering the sum $S_{6}$ we see that the summand for $j=0$ is zero and so
\begin{eqnarray*}
S_{6} & = & \sum_{j=1}^{[\frac{m}{2}]}\bigl(\begin{smallmatrix} m-2-j\\ j-1\end{smallmatrix}\bigr) \bigl(\begin{smallmatrix} a+b+c-n-2+j\\ a\end{smallmatrix}\bigr) \bigl(\begin{smallmatrix} b+c-l-j-1\\ b\end{smallmatrix}\bigr) \\
& = & \sum_{j=0}^{[\frac{m-2}{2}]}\bigl(\begin{smallmatrix} m-3-j\\ j\end{smallmatrix}\bigr) \bigl(\begin{smallmatrix} a+b+c-n-1+j\\ a\end{smallmatrix}\bigr) \bigl(\begin{smallmatrix} b+c-l-j-2\\ b\end{smallmatrix}\bigr).
\end{eqnarray*}

If $m-2$ is odd, then $[\frac{m-2}{2}]=[\frac{m-3}{2}]$. If $m-2$ is even, then $[\frac{m-2}{2}]=[\frac{m-3}{2}]+1$, but, as in the case of the sum $S_{4}$, for $m-2=2r$ ($r\geq 1$ since $m\geq 3$) the first factor of the summand obtained for $j=[\frac{m-2}{2}]=r$ equals ${r-1\choose r}=0$. We conclude that $S_{6}$ is equal to the sum
\[\sum_{j=0}^{[\frac{m-3}{2}]}\bigl(\begin{smallmatrix} m-3-j\\ j\end{smallmatrix}\bigr) \bigl(\begin{smallmatrix} a+b+c-n-1+j\\ a\end{smallmatrix}\bigr) \bigl(\begin{smallmatrix} b+c-l-j-2\\ b\end{smallmatrix}\bigr) 
=S(m-3,l+3,a,b,c+1)\equiv 0\imod 2,\]
by the induction hypothesis. Hence, $S(m,l,a,b,c)\equiv 0\imod 2$ and the proof of the Lemma \ref{l3} is completed.
\end{proof}

\begin{proposition}\label{p4} Let $m,l\in \mathbb{N}_{0}$ such that $m+l=n+2$. Then
\[g_{m,l}=\sum_{j=1}^{[\frac{m}{2}]}\bigl(\begin{smallmatrix} m-j\\ j\end{smallmatrix}\bigr) g_{m-2j,l+j}.\]
\end{proposition}

\begin{proof} According to Lemma \ref{l3}
\begin{eqnarray*}
g_{m,l} & = & \sum_{a+2b+3c=n+1+m+2l}\bigl(\begin{smallmatrix} a+b+c-m-l\\ a\end{smallmatrix}\bigr) \bigl(\begin{smallmatrix} b+c-l\\ b\end{smallmatrix}\bigr) w_{1}^{a}w_{2}^{b}w_{3}^{c}\\
& = & \sum_{a+2b+3c=n+1+m+2l}\bigl(\begin{smallmatrix} a+b+c-n-2\\ a\end{smallmatrix}\bigr) \bigl(\begin{smallmatrix} b+c-l\\ b \end{smallmatrix}\bigr) w_{1}^{a}w_{2}^{b}w_{3}^{c}\\
& = & \sum_{a+2b+3c=n+1+m+2l}\sum_{j=1}^{[\frac{m}{2}]}\bigl(\begin{smallmatrix} m-j\\ j\end{smallmatrix}\bigr) \bigl(\begin{smallmatrix} a+b+c-n-2+j\\ a\end{smallmatrix}\bigr) \bigl(\begin{smallmatrix} b+c-l-j\\ b\end{smallmatrix}\bigr) w_{1}^{a}w_{2}^{b}w_{3}^{c}\\
& = & \sum_{j=1}^{[\frac{m}{2}]}\bigl(\begin{smallmatrix} m-j\\ j\end{smallmatrix}\bigr) \sum_{a+2b+3c=n+1+m+2l}\bigl(\begin{smallmatrix} a+b+c-n-2+j\\ a\end{smallmatrix}\bigr) \bigl(\begin{smallmatrix} b+c-l-j\\ b\end{smallmatrix}\bigr) w_{1}^{a}w_{2}^{b}w_{3}^{c}.
\end{eqnarray*}

By Definition \ref{d2},
\[\sum_{a+2b+3c=n+1+m+2l}\bigl(\begin{smallmatrix} a+b+c-n-2+j\\ a\end{smallmatrix}\bigr) \bigl(\begin{smallmatrix} b+c-l-j\\ b\end{smallmatrix}\bigr) w_{1}^{a}w_{2}^{b}w_{3}^{c}\]
\[=\sum_{a+2b+3c=n+1+m-2j+2(l+j)}\bigl(\begin{smallmatrix} a+b+c-m+2j-l-j\\ a\end{smallmatrix}\bigr) \bigl(\begin{smallmatrix} b+c-(l+j)\\ b \end{smallmatrix}\bigr) w_{1}^{a}w_{2}^{b}w_{3}^{c}
=g_{m-2j,l+j}\]
and the proposition follows.
\end{proof}

In the following Proposition \ref{p5} we give some convenient presentations for $S$-polynomials of elements of $G$. Recall that (for a fixed monomial ordering) the $S$-polynomial of polynomials $f,g\in \mathbb{Z}_{2}[x_{1},x_{2},\dots,x_{k}]$ is given by
\[S(f,g)=\frac{L}{\mathrm{LT}(f)}\cdot f+\frac{L}{\mathrm{LT}(g)}\cdot g,\]
where $L=\mathrm{lcm}(\mathrm{LT}(f),\mathrm{LT}(g))$ denotes the least common multiple of $\mathrm{LT}(f)$ and $\mathrm{LT}(g)$.

\begin{lemma}\label{l5} Let $m,l\in \mathbb{N}_{0}$.

\noindent$\mathrm{(a)}$ If $r\in \mathbb{N}$ is such that $m+l<m+r+l\leq n+1$, then
\[S(g_{m,l},g_{m+r,l})=\sum_{i=0}^{r-1}w_{1}^{i}w_{2}^{r-1-i}(g_{m+2+i,l}+g_{m+i,l+1}).\]
$\mathrm{(b)}$ If $s\in \mathbb{N}$ is such that $m+l<m+l+s\leq n+1$, then
\[S(g_{m,l},g_{m,l+s})=\sum_{j=0}^{s-1}w_{1}^{j}w_{3}^{s-1-j}g_{m+1,l+1+j}.\]
$\mathrm{(c)}$ If $m+l\leq n+1$ and if $s\in \mathbb{N}$ is such that $m\geq s$, then
\[S(g_{m,l},g_{m-s,l+s})=\sum_{j=0}^{s-1}w_{2}^{j}w_{3}^{s-1-j}g_{m-1-j,l+2+j}.\]
\end{lemma}

\begin{proof} We shall prove the part (a) only. The proofs of (b) and (c) are similar. Observe that, according to Proposition \ref{p2}, $\textrm{LT}(g_{m,l})=w_{1}^{n+1-m-l}w_{2}^{m}w_{3}^{l}$ and $\textrm{LT}(g_{m+r,l})=w_{1}^{n+1-m-r-l}w_{2}^{m+r}w_{3}^{l}$. So we have \[\textrm{lcm}(\textrm{LT}(g_{m,l}),\textrm{LT}(g_{m+r,l}))=w_{1}^{n+1-m-l}w_{2}^{m+r}w_{3}^{l},\] implying
\[S(g_{m,l},g_{m+r,l})=w_{2}^{r}g_{m,l}+w_{1}^{r}g_{m+r,l}.\]

The proof is by induction on $r$. For $r=1$, we need to verify the equality $S(g_{m,l},g_{m+1,l})=g_{m,l+1}+g_{m+2,l}$. We have
\[S(g_{m,l},g_{m+1,l})=w_{2}g_{m,l}+w_{1}g_{m+1,l}\]
\[=\mkern-18mu\sum_{a+2b+3c= \atop= n+1+m+2l}\mkern-18mu\bigl(\begin{smallmatrix} a+b+c-m-l\\ a\end{smallmatrix}\bigr) \bigl(\begin{smallmatrix} b+c-l\\ b\end{smallmatrix}\bigr) w_{1}^{a}w_{2}^{b+1}w_{3}^{c}
+\mkern-36mu\sum_{a+2b+3c=\atop =n+1+m+1+2l}\mkern-18mu\bigl(\begin{smallmatrix} a+b+c-m-1-l\\ a\end{smallmatrix}\bigr) \bigl(\begin{smallmatrix} b+c-l\\ b \end{smallmatrix}\bigr) w_{1}^{a+1}w_{2}^{b}w_{3}^{c}\]
\[=\mkern-18mu\sum_{a+2b+3c=\atop =n+m+2l+3}\mkern-18mu\bigl(\begin{smallmatrix} a+b+c-m-l-1\\ a\end{smallmatrix}\bigr) \bigl(\begin{smallmatrix} b+c-l-1\\ b-1\end{smallmatrix}\bigr) w_{1}^{a}w_{2}^{b}w_{3}^{c}
+\mkern-18mu\sum_{a+2b+3c=\atop =n+m+2l+3}\mkern-18mu\bigl(\begin{smallmatrix} a+b+c-m-l-2\\ a-1\end{smallmatrix}\bigr) \bigl(\begin{smallmatrix} b+c-l\\ b\end{smallmatrix}\bigr) w_{1}^{a}w_{2}^{b}w_{3}^{c}.\]
Also, 
\[\bigl(\begin{smallmatrix} a+b+c-m-l-1\\ a\end{smallmatrix}\bigr) \bigl(\begin{smallmatrix} b+c-l-1\\ b-1\end{smallmatrix}\bigr) +\bigl(\begin{smallmatrix} a+b+c-m-l-2\\ a-1\end{smallmatrix}\bigr) \bigl(\begin{smallmatrix} b+c-l\\ b\end{smallmatrix}\bigr) \]
\[\equiv \bigl(\begin{smallmatrix} a+b+c-m-l-1\\ a\end{smallmatrix}\bigr) \bigl(\begin{smallmatrix} b+c-l-1\\ b-1\end{smallmatrix}\bigr) +\bigl(\begin{smallmatrix} a+b+c-m-l-1\\ a\end{smallmatrix}\bigr) \bigl(\begin{smallmatrix} b+c-l\\ b\end{smallmatrix}\bigr) \]
\[+\bigl(\begin{smallmatrix} a+b+c-m-l-1\\ a\end{smallmatrix}\bigr) \bigl(\begin{smallmatrix} b+c-l\\ b\end{smallmatrix}\bigr) +\bigl(\begin{smallmatrix} a+b+c-m-l-2\\ a-1\end{smallmatrix}\bigr) \bigl(\begin{smallmatrix} b+c-l\\ b\end{smallmatrix}\bigr) \]
\[\equiv \bigl(\begin{smallmatrix} a+b+c-m-l-1\\ a\end{smallmatrix}\bigr) \bigl(\begin{smallmatrix} b+c-l-1\\ b\end{smallmatrix}\bigr) +\bigl(\begin{smallmatrix} a+b+c-m-l-2\\ a\end{smallmatrix}\bigr) \bigl(\begin{smallmatrix} b+c-l\\ b\end{smallmatrix}\bigr) \]
and we obtain:
\begin{eqnarray*}
S(g_{m,l},g_{m+1,l})& = &
\sum_{a+2b+3c=\atop =n+m+2l+3}\bigl(\begin{smallmatrix} a+b+c-m-l-1\\ a\end{smallmatrix}\bigr) \bigl(\begin{smallmatrix} b+c-l-1\\ b \end{smallmatrix}\bigr) w_{1}^{a}w_{2}^{b}w_{3}^{c}\\
& + & \sum_{a+2b+3c=\atop=n+m+2l+3}\bigl(\begin{smallmatrix} a+b+c-m-l-2\\ a\end{smallmatrix}\bigr) \bigl(\begin{smallmatrix} b+c-l\\ b \end{smallmatrix}\bigr) w_{1}^{a}w_{2}^{b}w_{3}^{c}\\
& = & g_{m,l+1}+g_{m+2,l}.
\end{eqnarray*}

For the induction step we take $r\geq 2$ and calculate:
\[S(g_{m,l},g_{m+r,l})=w_{2}^{r}g_{m,l}+w_{1}^{r}g_{m+r,l}=w_{2}^{r}g_{m,l}+2w_{1}^{r-1}w_{2}g_{m+r-1,l}+w_{1}^{r}g_{m+r,l}\]
\[=w_{2}S(g_{m,l},g_{m+r-1,l})+w_{1}^{r-1}S(g_{m+r-1,l},g_{m+r,l})\]
\[=w_{2}\DS\sum_{i=0}^{r-2}w_{1}^{i}w_{2}^{r-2-i}(g_{m+2+i,l}+g_{m+i,l+1})+w_{1}^{r-1}(g_{m+r+1,l}+g_{m+r-1,l+1})\]
\[=\DS\sum_{i=0}^{r-1}w_{1}^{i}w_{2}^{r-1-i}(g_{m+2+i,l}+g_{m+i,l+1}),\]
by the induction hypothesis.
\end{proof}

Note that the previous lemma holds also for $r=0$ ($s=0$) since by definition $S(f,f)=0$ and the sums on the right hand side of the equalities are empty.

\begin{proposition}\label{p5} Let $m,l,r,s\in \mathbb{N}_{0}$.

\noindent$\mathrm{(a)}$ If $m+l<m+l+r+s\leq n+1$, then
\[S(g_{m,l},g_{m+r,l+s})=\]
\[=\DS\sum_{i=0}^{r-1}w_{1}^{s+i}w_{2}^{r-1-i}(g_{m+2+i,l+s}+g_{m+i,l+s+1})+
\DS\sum_{j=0}^{s-1}w_{1}^{j}w_{2}^{r}w_{3}^{s-1-j}g_{m+1,l+1+j}.\]
$\mathrm{(b)}$ If $l\geq s$, $r\geq s$ and $m+r+l-s\leq n+1$, then
\[S(g_{m,l},g_{m+r,l-s})=\]
\[=\DS\sum_{i=0}^{r-s-1}w_{1}^{i}w_{2}^{r-1-i}(g_{m+2+i,l}+g_{m+i,l+1})+
\DS\sum_{j=0}^{s-1}w_{1}^{r-s}w_{2}^{j}w_{3}^{s-1-j}g_{m+r-1-j,l-s+2+j}.\]
$\mathrm{(c)}$ If $l\geq s$, $r<s$ and $m+l\leq n+1$, then
\[S(g_{m,l},g_{m+r,l-s})=\]
\[=\DS\sum_{i=0}^{s-r-1}w_{1}^{i}w_{2}^{r}w_{3}^{s-r-1-i}g_{m+1,l-s+r+1+i}+
\DS\sum_{j=0}^{r-1}w_{2}^{j}w_{3}^{s-1-j}g_{m+r-1-j,l-s+2+j}.\]
\end{proposition}

\begin{proof} Again, we only prove the part (a), the proofs of (b) and (c) being completely analogous. Using Proposition \ref{p2}, we easily obtain that
\[\textrm{lcm}(\textrm{LT}(g_{m,l}),\textrm{LT}(g_{m+r,l+s}))=w_{1}^{n+1-m-l}w_{2}^{m+r}w_{3}^{l+s},\] and so
\[S(g_{m,l},g_{m+r,l+s})=w_{2}^{r}w_{3}^{s}g_{m,l}+w_{1}^{r+s}g_{m+r,l+s}.\]
Moving on, we have
\[S(g_{m,l},g_{m+r,l+s})=w_{2}^{r}w_{3}^{s}g_{m,l}+2w_{2}^{r}w_{1}^{s}g_{m,l+s}+w_{1}^{r+s}g_{m+r,l+s}\]
\[=w_{2}^{r}S(g_{m,l},g_{m,l+s})+w_{1}^{s}S(g_{m,l+s},g_{m+r,l+s})\]
\[=\DS\sum_{j=0}^{s-1}w_{1}^{j}w_{2}^{r}w_{3}^{s-1-j}g_{m+1,l+1+j}+
\DS\sum_{i=0}^{r-1}w_{1}^{s+i}w_{2}^{r-1-i}(g_{m+2+i,l+s}+g_{m+i,l+s+1}),\]
by parts (a) and (b) of Lemma \ref{l5}.
\end{proof}

Observe that in the previous proposition the $S$-polynomials of elements of $G$ are presented as some functions of polynomials $g_{m,l}$ where $m+l\leq n+2$. Those for which $m+l\leq n+1$ are elements of $G$ and those for which $m+l=n+2$ can be written as sums of elements of $G$ according to Proposition \ref{p4}. 

In order to prove that $G$ is a basis for the ideal $I_{3,n}$, i.e., $I_{G}=I_{3,n}$, we list the following equalities:
\begin{equation}\label{fff1} g_{m+2,l}=g_{m,l+1}+w_{2}g_{m,l}+w_{1}g_{m+1,l},
\end{equation}
\begin{equation}\label{fff2} g_{m+1,l+1}=w_{3}g_{m,l}+w_{1}g_{m,l+1},
\end{equation}
\begin{equation}\label{fff3} g_{m-1,l+2}=w_{3}g_{m,l}+w_{2}g_{m-1,l+1}.
\end{equation}
The first one is obtained in the proof of Lemma \ref{l5} as the induction base and the other two are actually parts (b) and (c) of that lemma for $s=1$. 

\begin{proposition}\label{pr11} $I_{G}=I_{3,n}$.
\end{proposition}
\begin{proof} According to Proposition \ref{p3}, $I_{3,n}\subseteq I_{G}$, so it remains to prove that $g\in I_{3,n}$ for all $g\in G$, i.e., $g_{m,l}\in I_{3,n}$ for all $m,l\in \mathbb{N}_{0}$ such that $m+l\leq n+1$. The proof is by induction on $m+l$. We already have that $g_{0,0}=\overline{w}_{n+1}\in I_{3,n}$. Also, in the proof of Proposition \ref{p3} we established that \[g_{1,0}=w_{1}g_{0,0}+\overline{w}_{n+2}=w_{1}\overline{w}_{n+1}+\overline{w}_{n+2}\in I_{3,n}\]
and that $g_{2,0}=w_{1}^{2}g_{0,0}+\overline{w}_{n+3}\in I_{3,n}$.
By formula (\ref{fff1}), $g_{2,0}=g_{0,1}+w_{2}g_{0,0}+w_{1}g_{1,0}$ and so
\[g_{0,1}=g_{2,0}+w_{2}g_{0,0}+w_{1}g_{1,0}\in I_{3,n}.\]
Therefore, $g_{m,l}\in I_{3,n}$ if $m+l\leq 1$.

Now, take $g_{m,l}\in G$ such that $m+l\geq 2$ and assume that $g_{\widetilde{m},\widetilde{l}}\in I_{3,n}$ if $\widetilde{m}+\widetilde{l}<m+l$. If $l=0$, then $m\geq 2$ and by formula (\ref{fff1}) we have
\[g_{m,0}=g_{m-2,1}+w_{2}g_{m-2,0}+w_{1}g_{m-1,0}\in I_{3,n}.\]
If $l=1$, formula (\ref{fff2}) gives us
\[g_{m,1}=w_{3}g_{m-1,0}+w_{1}g_{m-1,1}\in I_{3,n}.\]
Finally, if $l\geq 2$, we use formula (\ref{fff3}) and obtain
\[g_{m,l}=w_{3}g_{m+1,l-2}+w_{2}g_{m,l-1}\in I_{3,n},\]
by the induction hypothesis.
\end{proof}

\medskip

Our next task is to prove that $G$ is a Gr\"obner basis. We shall use the following definition and theorem (see \cite[p.\,219]{Becker}). It is assumed that a monomial ordering $\preceq$ on $\mathbb{Z}_{2}[x_{1},x_{2},\dots,x_{k}]$ is fixed.

\begin{definition}\label{de2} Let $F$ be a finite subset of $\mathbb{Z}_{2}[x_{1},x_{2},\dots,x_{k}]$, $f\in \mathbb{Z}_{2}[x_{1},x_{2},\dots,x_{k}]$ a nonzero polynomial and $t$ a fixed monomial. If $f$ can be written as a finite sum of the form $\DS\sum_{i}m_{i}f_{i}$ where $f_{i}\in F$ and $m_{i}\in \mathbb{Z}_{2}[x_{1},x_{2},\dots,x_{k}]$ are nonzero monomials such that $\mathrm{LT}(m_{i}f_{i})\preceq t$ for all $i$, we say that $\DS\sum_{i}m_{i}f_{i}$ is a \em{$t$-representation of $f$ with respect to $F$}.
\end{definition}

\begin{theorem}\label{th3} Let $F$ be a finite subset of $\mathbb{Z}_{2}[x_{1},x_{2},\dots,x_{k}]$, $0\notin F$. If for all $f_{1},f_{2}\in F$, $S(f_{1},f_{2})$ either equals zero or has a $t$-representation with respect to $F$ for some monomial
$t\prec \mathrm{lcm}(\mathrm{LT}(f_{1}),\mathrm{LT}(f_{2}))$, then $F$ is a Gr\"obner basis.
\end{theorem}

Using this theorem, we are able to prove that $G$ is a Gr\"obner basis for $I_{G}=I_{3,n}$.

\begin{theorem}\label{t3} Let $n\geq 3$. The set $G$ (see definitions {\em\ref{d2}} and {\em\ref{d3}}) is the reduced Gr\"obner basis for the ideal $I_{3,n}$ in $\mathbb{Z}_{2}[w_{1},w_{2},w_{3}]$ with respect to the grlex ordering $\preceq$.
\end{theorem}

\begin{proof} In order to apply Theorem \ref{th3}, we take two arbitrary elements of $G$, say $g_{m,l}$ and $g_{\widetilde{m},\widetilde{l}}$ ($g_{m,l}\neq g_{\widetilde{m},\widetilde{l}}$). Without loss of generality we may assume that either (i) $m<\widetilde{m}$ or else (ii) $m=\widetilde{m}$ and $l<\widetilde{l}$. We distinguish three cases.

\medskip

$1^{\circ}$ If condition (ii) holds or if $m<\widetilde{m}$ and $l\leq \widetilde{l}$, writing $\widetilde{m}=m+r$, $\widetilde{l}=l+s$, we have $m+l<m+l+r+s\leq n+1$, so the conditions of Proposition \ref{p5} (a) are satisfied implying \[S(g_{m,l},g_{\widetilde{m},\widetilde{l}})=S(g_{m,l},g_{m+r,l+s})\]
\[=\DS\sum_{i=0}^{r-1}w_{1}^{s+i}w_{2}^{r-1-i}(g_{m+2+i,l+s}+g_{m+i,l+s+1})+
\DS\sum_{j=0}^{s-1}w_{1}^{j}w_{2}^{r}w_{3}^{s-1-j}g_{m+1,l+1+j}.\]
If $m+l+r+s<n+1$, then all polynomials $g_{m,l}$ appearing in this expression are elements of $G$. If $m+l+r+s=n+1$, then $g_{m+r+1,l+s}$ and eventually $g_{m+1,l+s}$ (if $r=0$) are not in $G$. But, according to Proposition \ref{p4}, these two can be written as the sums of elements of $G$ and henceforth we consider these polynomials as the appropriate sums.

By Proposition \ref{p2} the leading terms of elements of $G$ all have the sum of the exponents equal to $n+1$. Therefore, the leading terms of the summands in the first sum all have the sum of the exponents $s+i+r-1-i+n+1=n+r+s$ and in the second $j+r+s-1-j+n+1=n+r+s$ too. We define $t=t(m,l,\widetilde{m},\widetilde{l})$ to be the maximum (with respect to $\preceq$) of all these leading terms. Hence, the above expression is a $t$-representation of $S(g_{m,l},g_{\widetilde{m},\widetilde{l}})$ w.r.t. $G$, $t$ has the sum of the exponents equal to $n+r+s$
and so \[t\prec w_{1}^{n+1-m-l}w_{2}^{m+r}w_{3}^{l+s}=
\textrm{lcm}(\textrm{LT}(g_{m,l}),\textrm{LT}(g_{\widetilde{m},\widetilde{l}})).\]

$2^{\circ}$ If $m<\widetilde{m}$, $l>\widetilde{l}$ and $\widetilde{m}-m\geq l-\widetilde{l}$, writing $\widetilde{m}=m+r$, $\widetilde{l}=l-s$, we have $l\geq s$, $r\geq s$ and $m+r+l-s\leq n+1$, i.e., the conditions of part (b) of Proposition \ref{p5} are satisfied and consequently
\[S(g_{m,l},g_{\widetilde{m},\widetilde{l}})=S(g_{m,l},g_{m+r,l-s})\]
\[=\DS\sum_{i=0}^{r-s-1}w_{1}^{i}w_{2}^{r-1-i}(g_{m+2+i,l}+g_{m+i,l+1})+
\DS\sum_{j=0}^{s-1}w_{1}^{r-s}w_{2}^{j}w_{3}^{s-1-j}g_{m+r-1-j,l-s+2+j}.\]
As in the previous case, for $m+r+l-s=n+1$ the polynomials $g_{m+r-s+1,l}$ and $g_{m+r-1-j,l-s+2+j}$ ($j=\overline{0,s-1}$) are treated as sums of elements of $G$ (obtained in Proposition \ref{p4}).

Again, we define $t$ to be the maximum of all leading terms in this expression and so we have a $t$-representation of $S(g_{m,l},g_{\widetilde{m},\widetilde{l}})$ w.r.t. $G$. Since the sum of the exponents in the leading terms is equal to $i+r-1-i+n+1=n+r$, i.e., $r-s+j+s-1-j+n+1=n+r$, we have
\[t\prec w_{1}^{n+1-m-l}w_{2}^{m+r}w_{3}^{l}=\textrm{lcm}(\textrm{LT}(g_{m,l}),\textrm{LT}(g_{\widetilde{m},\widetilde{l}})).\]

$3^{\circ}$ Finally, if $m<\widetilde{m}$, $l>\widetilde{l}$ and $\widetilde{m}-m<l-\widetilde{l}$, again we put $\widetilde{m}=m+r$, $\widetilde{l}=l-s$. In this case, $l\geq s$, $r<s$ and $m+l\leq n+1$, hence we may apply Proposition \ref{p5} (c) and obtain
\[S(g_{m,l},g_{\widetilde{m},\widetilde{l}})=S(g_{m,l},g_{m+r,l-s})\]
\[=\DS\sum_{i=0}^{s-r-1}w_{1}^{i}w_{2}^{r}w_{3}^{s-r-1-i}g_{m+1,l-s+r+1+i}+
\DS\sum_{j=0}^{r-1}w_{2}^{j}w_{3}^{s-1-j}g_{m+r-1-j,l-s+2+j}.\]
Considering this case as the previous two, we observe that the sum of the exponents in the leading terms is $i+r+s-r-1-i+n+1=n+s$, i.e., $j+s-1-j+n+1=n+s$. Defining $t$ as before, we have
\[t\prec w_{1}^{n+1-m-l+s-r}w_{2}^{m+r}w_{3}^{l}=\textrm{lcm}(\textrm{LT}(g_{m,l}),\textrm{LT}(g_{\widetilde{m},\widetilde{l}})).\]

Therefore, by Theorem \ref{th3} we conclude that $G$ is a Gr\"obner basis. According to Proposition \ref{p2}, all terms in $g_{m,l}\in G$, except the leading one, have the sum of the exponents $<n+1$ and hence they cannot be divisible by any leading term in $G$. This means that $G$ is the reduced Gr\"obner basis for $I_{3,n}$.
\end{proof}

Since $\textrm{LT}(g_{m,l})=w_{1}^{n+1-m-l}w_{2}^{m}w_{3}^{l}$ ($m,l\in \mathbb{N}_{0}$, $m+l\leq n+1$), we see that the set of all leading terms in $G$ is the set of all monomials with the sum of the exponents equal to $n+1$. Therefore, a monomial $w_{1}^{a}w_{2}^{b}w_{3}^{c}\in \mathbb{Z}_{2}[w_{1},w_{2},w_{3}]$ is not divisible by any of these leading terms if and only if $a+b+c\leq n$. By observing this equivalence we have proved the following corollary.

\begin{corollary}\label{c2} Let $n\geq 3$. If $w_{i}$ is the $i$-th Stiefel-Whitney class of the canonical vector bundle $\gamma_{3}$ over $G_{3,n}$, then the set $\{ w_{1}^{a}w_{2}^{b}w_{3}^{c} \mid a+b+c\leq n\}$ is a vector space basis for $H^{*}(G_{3,n};\mathbb{Z}_{2})$.
\end{corollary}

\medskip

Let us now calculate a few elements of the Gr\"obner basis $G$. By Proposition \ref{p2}, excluding the leading term $\textrm{LT}(g_{m,l})=w_{1}^{n+1-m-l}w_{2}^{m}w_{3}^{l}$, the monomial $w_{1}^{a}w_{2}^{b}w_{3}^{c}$ appears in $g_{m,l}$ only if $a+b+c<n+1$, so then we have $c\leq b+c\leq a+b+c\leq n$ and we conclude that $a+2b+3c\leq 3n$. Since $a+2b+3c$ must be equal to $n+1+m+2l$, we see that if $n+1+m+2l>3n$ (i.e., $m+2l>2n-1$) then $g_{m,l}=\textrm{LT}(g_{m,l})=w_{1}^{n+1-m-l}w_{2}^{m}w_{3}^{l}$. In particular, we have the equalities:
\[g_{0,n+1}=w_{3}^{n+1}; \qquad g_{0,n}=w_{1}w_{3}^{n}; \qquad g_{1,n}=w_{2}w_{3}^{n}.\]
Starting from these three, we can calculate the polynomials $g_{m,n-1}$, $g_{m,n-2}$, $g_{m,n-3}$ etc. using formulas (\ref{fff1}), (\ref{fff2}) and (\ref{fff3}). Namely, from (\ref{fff2}) we have $w_{3}g_{0,n-1}=w_{1}g_{0,n}+g_{1,n}=w_{1}^{2}w_{3}^{n}+w_{2}w_{3}^{n}$, so
\[g_{0,n-1}=w_{1}^{2}w_{3}^{n-1}+w_{2}w_{3}^{n-1}.\]
Using (\ref{fff3}), one obtains $w_{3}g_{1,n-1}=w_{2}g_{0,n}+g_{0,n+1}=w_{1}w_{2}w_{3}^{n}+w_{3}^{n+1}$, implying:
\[g_{1,n-1}=w_{1}w_{2}w_{3}^{n-1}+w_{3}^{n}.\]
Applying formula (\ref{fff1}), we have \[g_{2,n-1}=g_{0,n}+w_{2}g_{0,n-1}+w_{1}g_{1,n-1}\]
\[=w_{1}w_{3}^{n}+w_{1}^{2}w_{2}w_{3}^{n-1}+w_{2}^{2}w_{3}^{n-1}+w_{1}^{2}w_{2}w_{3}^{n-1}+w_{1}w_{3}^{n}
=w_{2}^{2}w_{3}^{n-1}.\]
Continuing in the same manner, we get the following table (Table \ref{tab1} on the next page) containing the polynomials $g_{m,l}\in G$ for $l\geq n-5$ (the leading terms are marked).

\begin{sidewaystable}[hbtp]
\vskip11cm
\caption{$g_{m,l}$} \label{tab1}
\centering
\resizebox{1.0\textheight}{!}{%
\begin{tabular}{r|p{37mm}|p{37mm}|p{33mm}|p{32mm}|l|l|l|}
\backslashbox{$m$}{$l$} & $n-5$ & $n-4$ & $n-3$ & $n-2$ & $n-1$ & $n$ & $n+1$\\ \toprule
$0$    & \fbox{$w_1^6w_3^{n-5}$}$+w_1^4w_2w_3^{n-5}$ $+w_2^3w_3^{n-5}+w_3^{n-3}$ &
\fbox{$w_1^5w_3^{n-4}$}$+w_1^2w_3^{n-3}$ $+w_1w_2^2w_3^{n-4}$ &
\fbox{$w_1^4w_3^{n-3}$}$+w_1^2w_2w_3^{n-3}$ $+w_2^2w_3^{n-3}$ &
\fbox{$w_1^3w_3^{n-2}$}$+w_3^{n-1}$&
\fbox{$w_1^2w_3^{n-1}$}$+w_2w_3^{n-1}$ &
\fbox{$w_1w_3^n$} &\fbox{$w_3^{n+1}$} \\ \midrule
$1$ & \fbox{$w_1^5w_2w_3^{n-5}$}$+w_1^4w_3^{n-4}$ $+w_1w_2^3w_3^{n-5}+w_2^2w_3^{n-4}$ & \fbox{$w_1^4w_2w_3^{n-4}$}$+w_1^3w_3^{n-3}$ $+w_1^2w_2^2w_3^{n-4}+w_2^3w_3^{n-4}$ $+w_3^{n-2}$ & \fbox{$w_1^3w_2w_3^{n-3}$}$+w_1^2w_3^{n-2}$ & \fbox{$w_1^2w_2w_3^{n-2}$}$+w_1w_3^{n-1}$ $+w_2^2w_3^{n-2}$ & \fbox{$w_1w_2w_3^{n-1}$}$+w_3^n$ & \fbox{$w_2w_3^n$} &  \\ \midrule
$2$ & \fbox{$w_1^4w_2^2w_3^{n-5}$}$+w_1^2w_2^3w_3^{n-5}$ $+w_1^2w_3^{n-3}+w_2^4w_3^{n-5}$ $+w_2w_3^{n-3}$ & \fbox{$w_1^3w_2^2w_3^{n-4}$}$+w_1w_3^{n-2}$ $+w_2^2w_3^{n-3}$ & \fbox{$w_1^2w_2^2w_3^{n-3}$}$+w_2^3w_3^{n-3}$ $+w_3^{n-1}$ & \fbox{$w_1w_2^2w_3^{n-2}$} & \fbox{$w_2^2w_3^{n-1}$} & & \\ \midrule
$3$ & \fbox{$w_1^3w_2^3w_3^{n-5}$}$+w_1^2w_2^2w_3^{n-4}$ $+w_1w_2w_3^{n-3}+w_3^{n-2}$ & \fbox{$w_1^2w_2^3w_3^{n-4}$}$+w_1w_2^2w_3^{n-3}$ $+w_2^4w_3^{n-4}+w_2w_3^{n-2}$ & \fbox{$w_1w_2^3w_3^{n-3}$}$+w_2^2w_3^{n-2}$ & \fbox{$w_2^3w_3^{n-2}$}$+w_3^n$ & & & \\ \midrule
$4$ & \fbox{$w_1^2w_2^4w_3^{n-5}$}$+w_2^5w_3^{n-5}$ & \fbox{$w_1w_2^4w_3^{n-4}$}$+w_3^{n-1}$& \fbox{$w_2^4w_3^{n-3}$}$+w_2w_3^{n-1}$ & & & & \\ \midrule
$5$ & \fbox{$w_1w_2^5w_3^{n-5}$}$+w_2^4w_3^{n-4}$ & \fbox{$w_2^5w_3^{n-4}$}$+w_1w_3^{n-1}$ & & & & &  \\ \midrule
$6$ & \fbox{$w_2^6w_3^{n-5}$}$+w_3^{n-1}$  & & & & & & \\ \bottomrule
\end{tabular}}
\end{sidewaystable}

In addition to the elements in the table, we write out a few more which will appear in our calculations and which can be obtained from the table by multiple applications of formulas (\ref{fff1})-(\ref{fff3}):
\[
\begin{array}{rclr}
g_{8,n-9} & = & w_{1}^{2}w_{2}^{8}w_{3}^{n-9}+w_{2}^{9}w_{3}^{n-9}+w_{2}^{3}w_{3}^{n-5}+w_{3}^{n-3} & (n\geq 9); \\
g_{10,n-9} & = & w_{2}^{10}w_{3}^{n-9}+w_{1}^{2}w_{2}^{3}w_{3}^{n-5}+w_{1}^{2}w_{3}^{n-3}+w_{2}w_{3}^{n-3} & (n\geq 9); \\
g_{10,n-10} & = & w_{1}w_{2}^{10}w_{3}^{n-10}+w_{1}^{3}w_{3}^{n-4}+w_{1}^{2}w_{2}^{2}w_{3}^{n-5}+w_{3}^{n-3} & (n\geq 10); \\
g_{12,n-12} & = & w_{1}w_{2}^{12}w_{3}^{n-12}+w_{1}^{4}w_{3}^{n-5}+w_{2}^{8}w_{3}^{n-9}+w_{1}w_{3}^{n-4} & (n\geq 12).
\end{array}
\]

\section{Immersions}
\label{immer}

In order to construct the immersions of Grassmannians $G_{3,n}$ into Euclidean spaces, we recall the theorem of Hirsch (\cite{Hirsch}) which states that a smooth compact $m$-manifold $M^{m}$ immerses in $\mathbb{R}^{m+l}$ if and only if the classifying map $f_{\nu}:M^{m}\rightarrow BO$ of the stable normal bundle $\nu$ of $M^{m}$ lifts up to $BO(l)$.
\[\bfig
\morphism<600,0>[M^{m}`BO;f_{\nu}]
\morphism(600,500)|r|<0,-500>[BO(l)`BO;p]
\morphism/-->/<600,500>[M^{m}`BO(l);]
\efig\]

Let $\mathrm{imm}(M^{m})$ denotes the least integer $d$ such that $M^{m}$ immerses into $\mathbb{R}^{d}$. By Hirsch's theorem, if $w_{k}(\nu )\neq 0$ then $\mathrm{imm}(M^{m})\geq m+k$.

\medskip

As in  Corollary \ref{c2}, let $w_{i}$ be the $i$-th Stiefel-Whitney class of the canonical vector bundle $\gamma_{3}$ over $G_{3,n}$ ($n\geq 3$) and let $r$ be the (unique) integer such that $2^{r+1}<3n<2^{r+2}$, i.e., $\frac{2}{3}\cdot 2^{r}<n<\frac{4}{3}\cdot 2^{r}$. It is well known (see \cite[p.\,183]{Oproiu}) that for the stable normal bundle $\nu$ of $G_{3,n}$ one has:
\begin{equation}\label{o1} w(\nu )=(1+w_{1}^{4}+w_{2}^{2}+w_{1}^{2}w_{2}^{2}+w_{3}^{2})(1+w_{1}+w_{2}+w_{3})^{2^{r+1}-n-3}.
\end{equation}
For $n\leq 2^{r}-3$, by the result of Stong (\cite{Stong}) $\mathrm{ht}(w_{1})=2^{r}-1$ and by the result of Dutta and Khare (\cite{Dutta}) $\mathrm{ht}(w_{2})\leq 2^{r}-1$. Also, $w_{3}^{2^{r}}=0$ since $3\cdot 2^{r}>3\cdot (2^{r}-3)\geq 3n=\mathrm{dim}(G_{3,n})$ and we have that $(1+w_{1}+w_{2}+w_{3})^{2^{r}}=1$. This means that in this case ($\frac{2}{3}\cdot 2^{r}<n\leq 2^{r}-3$) formula (\ref{o1}) simplifies to
\begin{equation}\label{o2} w(\nu )=(1+w_{1}^{4}+w_{2}^{2}+w_{1}^{2}w_{2}^{2}+w_{3}^{2})(1+w_{1}+w_{2}+w_{3})^{2^{r}-n-3}.
\end{equation}

\begin{theorem}[Oproiu \cite{Oproiu}]\label{t10} For the immersion dimension of $G_{3,n}$ we have:
\begin{enumerate}
\item[(a)] If $2^{r}\leq n<\frac{4}{3}\cdot 2^{r}$, then $\mathrm{imm}(G_{3,n})\geq 6\cdot 2^{r}-3$.
\item[(b)] If $\frac{2}{3}\cdot 2^{r}<n\leq 2^{r}-3$, then $\mathrm{imm}(G_{3,n})\geq 3\cdot 2^{r}-3$.
\end{enumerate}
\end{theorem}

\begin{proof} (a) By formula (\ref{o1}) above, the top degree class that appears in $w(\nu )$ is $w_{6\cdot 2^{r}-3n-3}(\nu )=w_{1}^{2}w_{2}^{2}w_{3}^{2^{r+1}-n-3}+w_{3}^{2^{r+1}-n-1}$. Now, $2^{r+1}-n-1\leq 2^{r}-1<n$, so $w_{3}^{2^{r+1}-n-1}$ is an element of the base from Corollary \ref{c2}. Also, $2+2+2^{r+1}-n-3=2^{r+1}-n+1\leq 2^{r}+1$. Hence, if $n\geq 2^{r}+1$ we have that $w_{6\cdot 2^{r}-3n-3}(\nu )$ is a sum of two distinct basis elements and we conclude that $w_{6\cdot 2^{r}-3n-3}(\nu )\neq 0$.

For $n=2^{r}$, we consider the element $g_{2,2^{r}-3} =w_{1}^{2}w_{2}^{2}w_{3}^{2^{r}-3}+w_{2}^{3}w_{3}^{2^{r}-3}+w_{3}^{2^r-1}$ of the Gr\"obner basis $G$ from Theorem \ref{t3} (see Table \ref{tab1}).
This implies that
\[w_{6\cdot 2^{r}-3n-3}(\nu )=w_{3\cdot 2^{r}-3}(\nu )=w_{1}^{2}w_{2}^{2}w_{3}^{2^{r}-3}+w_{3}^{2^{r}-1}
=g_{2,2^{r}-3}+w_{2}^{3}w_{3}^{2^{r}-3}=w_{2}^{3}w_{3}^{2^{r}-3},\]
which is nonzero in cohomology by Corollary \ref{c2} (the remainder of the division of $w_{6\cdot 2^{r}-3n-3}(\nu )$ by $G$ is nontrivial). Therefore, we have that
\[\mathrm{imm}(G_{3,n})\geq \mathrm{dim}(G_{3,n})+6\cdot 2^{r}-3n-3=3n+6\cdot 2^{r}-3n-3=6\cdot 2^{r}-3.\]

(b) Using formula (\ref{o2}), we obtain that, in this case, the top class in $w(\nu )$ is $w_{3\cdot 2^{r}-3n-3}(\nu )=w_{1}^{2}w_{2}^{2}w_{3}^{2^{r}-n-3}+w_{3}^{2^{r}-n-1}$. The sums of the exponents in these two monomials are $\leq 2^{r}-n+1<2^{r}-\frac{2}{3}\cdot 2^{r}+1=\frac{1}{3}\cdot 2^{r}+1<\frac{n}{2}+1<n$. This means that $w_{3\cdot 2^{r}-3n-3}(\nu )$ is a sum of two distinct basis elements (from Corollary \ref{c2}). Hence, $w_{3\cdot 2^{r}-3n-3}(\nu )\neq 0$ and we conclude that
\[\mathrm{imm}(G_{3,n})\geq 3n+3\cdot 2^{r}-3n-3=3\cdot 2^{r}-3\]
completing the proof of the theorem.
\end{proof}

\begin{remark} In \cite{Oproiu}, Oproiu has also proved that $\mathrm{imm}(G_{3,2^{r}-2})\geq 4\cdot 2^{r}-3$, $\mathrm{imm}(G_{3,2^{r}-1})\geq 5\cdot 2^{r}-3$ ($r\geq 3$) and $\mathrm{imm}(G_{3,3})\geq 15$. The fact $\mathrm{imm}(G_{3,3})\geq 15$ is easily obtained by our method. Likewise, using the Gr\"obner basis $G$, one can verify that $\mathrm{imm}(G_{3,2^{r}-2})\geq 4\cdot 2^{r}-3$, but since the proof requires a lot of calculation, we have decided to omit it.
\end{remark}

In order to shorten the upcoming calculations, we give two equalities concerning the action of the Steenrod algebra $\mathcal{A}_{2}$ on $H^{*}(G_{3,n};\mathbb{Z}_{2})$ which can be obtained using the basic properties of $\mathcal{A}_{2}$ and formulas of Wu and Cartan. It is understood that $a$, $b$ and $c$ are nonnegative integers.

\begin{equation}\label{fsa1} Sq^{1}(w_{1}^{a}w_{2}^{b}w_{3}^{c})=(a+b+c)w_{1}^{a+1}w_{2}^{b}w_{3}^{c}+bw_{1}^{a}w_{2}^{b-1}w_{3}^{c+1},
\end{equation}
\begin{equation}\label{fsa2} Sq^{2}(w_{1}^{a}w_{2}^{b}w_{3}^{c})={a+b+c\choose 2}w_{1}^{a+2}w_{2}^{b}w_{3}^{c}+b(a+c)w_{1}^{a+1}w_{2}^{b-1}w_{3}^{c+1}
\end{equation}
\[+(b+c)w_{1}^{a}w_{2}^{b+1}w_{3}^{c}+{b\choose 2}w_{1}^{a}w_{2}^{b-2}w_{3}^{c+2}.\]

Now, we turn to the proof of Theorem \ref{theorem2}.

\begin{lemma}\label{l11} Let $n\equiv 0\imod 4$. If $\nu$ is the stable normal bundle of $G_{3,n}$, then
\begin{enumerate}
\item[$\mathrm{(a)}$] $w_{i}(\nu )=0$ for $i\geq 3n-2$;
\item[$\mathrm{(b)}$] $w_{2}(\nu )=w_{2}$.
\end{enumerate}
\end{lemma}
\begin{proof} As above, let $r$ be the integer such that $\frac{2}{3}\cdot 2^{r}<n<\frac{4}{3}\cdot 2^{r}$.

If $n\geq 2^{r}$, then $6\cdot 2^{r}\leq 6n$ implying $6\cdot 2^{r}-3n-3\leq 3n-3$. As we have already noticed in the proof of the previous theorem, the top class in the expression (\ref{o1}) is of degree $6\cdot 2^{r}-3n-3$ and the previous inequality proves (a).

If $n<2^{r}$, then $n$ actually must be $<2^{r}-3$ (since $n\equiv 0\imod 4$) and by looking at formula (\ref{o2})
we see that the top class there is in degree $3\cdot 2^{r}-3n-3$ and, since $n>\frac{2}{3}\cdot 2^{r}$, we have that $3\cdot 2^{r}<4\cdot 2^{r}<6n$ implying $3\cdot 2^{r}-3n-3<3n-3$. This proves (a).

\medskip

Using the fact that $2^{r+1}-n-3\equiv 1\imod 4$ (since $n\equiv 0\imod 4$), from formula (\ref{o1}) we directly read off:
\[w_{2}(\nu )={2^{r+1}-n-3\choose 2}w_{1}^{2}+(2^{r+1}-n-3)w_{2}=w_{2},\]
obtaining (b).
\end{proof}

\noindent{\bf Proof of Theorem \ref{theorem2}.} Let $f_{\nu}:G_{3,n}\rightarrow BO$ be the classifying map for the stable normal bundle $\nu$ of $G_{3,n}$. In order to show that $f_{\nu}$ can be lifted up to $BO(3n-3)$, we use the $3n$-MPT for the fibration
$p:BO(3n-3)\rightarrow BO$ which can be constructed by the method of Gitler and Mahowald (\cite{Gitler}) using the result of Nussbaum (\cite{Nussbaum}) who has proved that their method is applicable to the fibrations $p:BO(l)\rightarrow BO$ when $l$ is odd. The tower is presented in the following diagram ($K_{m}$ stands for the Eilenberg-MacLane space $K(\mathbb{Z}_{2},m)$).
\[\bfig
 \morphism<900,0>[G_{3,n}`BO;f_{\nu}]
 \morphism(900,0)<1100,0>[BO`K_{3n-2}\times K_{3n};w_{3n-2}\times w_{3n}]
 \morphism(900,500)|r|<0,-500>[E_{1}`BO;q_{1}]
 \morphism(900,500)<1100,0>[E_{1}`K_{3n-1}\times K_{3n};k_{1}^{2}\times k_{2}^{2}]
 \morphism(900,1000)|r|<0,-500>[E_{2}`E_{1};q_{2}]
 \morphism(900,1000)<1100,0>[E_{2}`K_{3n};k_{1}^{3}]
 \morphism(900,1500)|r|<0,-500>[E_{3}`E_{2};q_{3}]
 \morphism/-->/<900,500>[G_{3,n}`E_{1};g]
 \morphism/-->/<900,1000>[G_{3,n}`E_{2};h]
 \efig\]
The relations that produce $k$-invariants are given in the following table.
\begin{table}[ht]
\label{eqtable}
\renewcommand\arraystretch{1.5}
\noindent\[
\begin{array}{|l|}
\hline
k_{1}^{2}: \quad (Sq^{2}+w_{2})w_{3n-2}=0\\
\hline
k_{2}^{2}: \quad (Sq^{2}+w_{1}^{2}+w_{2})Sq^{1}w_{3n-2}+Sq^{1}w_{3n}=0\\
\hline
k_{1}^{3}: \quad (Sq^{2}+w_{2})k_{1}^{2}+Sq^{1}k_{2}^{2}=0\\
\hline
\end{array}
\]
\end{table}

This is $3n$-MPT and since $\mathrm{dim}(G_{3,n})=3n$, it suffices to lift $f_{\nu}$ up to $E_{3}$.

\medskip

By Lemma \ref{l11} (a), $f_{\nu}^{*}(w_{3n-2})=w_{3n-2}(\nu )=0$, $f_{\nu}^{*}(w_{3n})=w_{3n}(\nu )=0$, so $f_{\nu}$ can be lifted up to $E_{1}$.

\medskip

Now, we show that we can choose a lifting $g:G_{3,n}\rightarrow E_{1}$ of $f_{\nu}$ which lifts up to $E_{2}$. We use Lemma \ref{l11} (b), formula (\ref{fsa2}) and Gr\"obner basis $G$ from Theorem \ref{t3} to calculate:
\[(Sq^{2}+w_{2}(\nu ))(w_{3}^{n-1})=Sq^{2}w_{3}^{n-1}+w_{2}w_{3}^{n-1}={n-1\choose 2}w_{1}^{2}w_{3}^{n-1}+(n-1)w_{2}w_{3}^{n-1}\]
\[+w_{2}w_{3}^{n-1}=w_{1}^{2}w_{3}^{n-1}=g_{0,n-1}+w_{2}w_{3}^{n-1}=w_{2}w_{3}^{n-1}.\]
By Corollary \ref{c2}, $w_{2}w_{3}^{n-1}\neq 0$ in $H^{3n-1}(G_{3,n};\mathbb{Z}_{2})\cong \mathbb{Z}_{2}$. Hence, the indeterminacy of $k_{1}^{2}$ is all of $H^{3n-1}(G_{3,n};\mathbb{Z}_{2})$, so we can choose $g$ such that $g^{*}(k_{1}^{2})=0$. Also,
\[Sq^{1}(w_{2}w_{3}^{n-1})=nw_{1}w_{2}w_{3}^{n-1}+w_{3}^{n}=w_{3}^{n}\neq0,\]
by (\ref{fsa1}) and Corollary \ref{c2}. By looking at the relation that produces $k_{2}^{2}$, we see that we can vary $g$ (by choosing appropriate class in $H^{3n-1}(G_{3,n};\mathbb{Z}_{2})$) such that $g^{*}(k_{1}^{2})=g^{*}(k_{2}^{2})=0$. This means that $g$ (and then also $f_{\nu}$) lifts up to $E_{2}$.

\medskip

Finally, since $Sq^{1}(w_{2}w_{3}^{n-1})=w_{3}^{n}\neq 0$, we conclude that the indeterminacy of $k_{1}^{3}$ is all of $H^{3n}(G_{3,n};\mathbb{Z}_{2})\cong \mathbb{Z}_{2}$, so the appropriate lifting $h:G_{3,n}\rightarrow E_{2}$ of $f_{\nu}$ lifts up to $E_{3}$ and the proof of the theorem is completed.  \hfill $\Box$

\medskip

In order to prove Theorem \ref{theorem3}, we focus on the case $n\equiv 6\imod 8$. We need the following lemma.

\begin{lemma}\label{l21} Let $n\equiv 6\imod 8$. If $\nu$ is the stable normal bundle of $G_{3,n}$,
\begin{enumerate}
\item[$\mathrm{(a)}$] $w_{3n-4}(\nu )=0$;
\item[$\mathrm{(b)}$] $w_{3n-2}(\nu )=0$;
\item[$\mathrm{(c)}$] $w_{2}(\nu )=w_{1}^{2}+w_{2}$.
\end{enumerate}
\end{lemma}

\begin{proof} Let $r\geq 3$ be the integer such that $2^{r+1}<3n<2^{r+2}$. If $n\geq 2^{r}$, then $n\geq 2^{r}+6$ (since $n\equiv 6\imod 8$) and so $2^{r+1}\leq 2n-12$. The top class in the expression (\ref{o1}) is in degree $6+3\cdot (2^{r+1}-n-3)\leq 6+3\cdot (n-15)=3n-39$ and obviously, we have (a) and (b).

If $n<2^{r}$ and $n\neq 2^{r}-2$, then $n$ must be $<2^{r}-2$, so formula (\ref{o2}) holds. The top class there is in degree $6+3\cdot (2^{r}-n-3)$ and, since $3n>2^{r+1}$, we have that $2^{r}<\frac{3}{2}\cdot n$, implying $6+3\cdot (2^{r}-n-3)<6+3\cdot \frac{n-6}{2}\leq 3n-12$ and again, $w_{3n-4}(\nu )$ and $w_{3n-2}(\nu )$ must be trivial.

To prove (a) and (b) it remains to consider the case $n=2^{r}-2$. From formula (\ref{o1}), we have
\[w(\nu )=((1+w_{1}^{2}+w_{2}+w_{1}w_{2})^{2}+w_{3}^{2})(1+w_{1}+w_{2}+w_{3})^{2^{r}-1}\]
\[=((1+w_{1})^{2}(1+w_{1}+w_{2})^{2}+w_{3}^{2})\sum_{j=0}^{2^{r}-1}(1+w_{1}+w_{2})^{j}w_{3}^{2^{r}-1-j}\]
\[=(1+w_{1}^{2})\sum_{j=1}^{n+1}(1+w_{1}+w_{2})^{j+2}w_{3}^{n+1-j}+\sum_{j=1}^{n+1}(1+w_{1}+w_{2})^{j}w_{3}^{n+3-j}\]
\[=(1+w_{1}^{2})\sum_{j=3}^{n+3}(1+w_{1}+w_{2})^{j}w_{3}^{n+3-j}+\sum_{j=3}^{n+1}(1+w_{1}+w_{2})^{j}w_{3}^{n+3-j}\]
\[=(1+w_{1}+w_{2})^{n+2}(1+w_{1}+w_{2}+w_{3})+\sum_{j=3}^{n+3}w_{1}^{2}(1+w_{1}+w_{2})^{j}w_{3}^{n+3-j}\]
\[=(1\!+\!w_{1}\!+\!w_{2})^{n+2}(1\!+\!w_{1}\!+\!w_{2}\!+\!w_{3})
\!+\!\sum_{j=3}^{n+3}\!\sum_{a+b\leq j}\!\bigl(\begin{smallmatrix} j\\ a+b\end{smallmatrix}\bigr) \bigl(\begin{smallmatrix} a+b\\ a\end{smallmatrix}\bigr) w_{1}^{a+2}w_{2}^{b}w_{3}^{n+3-j}.\]
Dutta and Khare (\cite{Dutta}) proved that, in this case, $\mathrm{ht}(w_{2})=2^{r}-1$. Thus,
\[(1+w_{1}+w_{2})^{n+2}=(1+w_{1}+w_{2})^{2^{r}}=1+w_{1}^{2^{r}}+w_{2}^{2^{r}}=1+w_{1}^{2^{r}}=1+w_{1}^{n+2}.\]
Finally, we obtain the following expression for $w(\nu )$:
\[w(\nu )=(1\!+\!w_{1}^{n+2})(1\!+\!w_{1}\!+\!w_{2}\!+\!w_{3})
+\sum_{j=3}^{n+3}\sum_{a+b\leq j}\bigl(\begin{smallmatrix} j\\ a+b\end{smallmatrix}\bigr) \bigl(\begin{smallmatrix} a+b\\ a\end{smallmatrix}\bigr) w_{1}^{a+2}w_{2}^{b}w_{3}^{n+3-j}.\]
Using this formula, after some tedious calculation (which we omit), one gets
\[w_{3n-4}(\nu )=w_{1}^{2}w_{3}^{n-2}+w_{1}^{8}w_{3}^{n-4}+w_{1}^{6}w_{2}w_{3}^{n-4}+w_{1}^{2}w_{2}^{3}w_{3}^{n-4}\]
\[+w_{1}^{4}w_{2}^{8}w_{3}^{n-8}+w_{1}^{2}w_{2}^{9}w_{3}^{n-8}+w_{1}^{2}w_{2}^{12}w_{3}^{n-10}+w_{1}^{2}w_{2}^{15}w_{3}^{n-12}\]
with the note that for $r=3$, i.e., $n=6$, only the first four summands appear. We are going to prove that both sum of the first four and sum of the second four summands are equal to zero. We use elements of Gr\"obner basis $G$ from Theorem \ref{t3}.
\[w_{1}^{2}w_{3}^{n-2}+w_{1}^{8}w_{3}^{n-4}+w_{1}^{6}w_{2}w_{3}^{n-4}+w_{1}^{2}w_{2}^{3}w_{3}^{n-4}=w_{1}^{2}w_{3}^{n-2}\]
\[+w_{1}^{3}(g_{0,n-4}+w_{1}^{2}w_{3}^{n-3}+w_{1}w_{2}^{2}w_{3}^{n-4})+w_{1}^{2}(g_{1,n-4}+w_{1}^{3}w_{3}^{n-3}
+w_{1}^{2}w_{2}^{2}w_{3}^{n-4}+w_{3}^{n-2})\]
and since $g_{m,l}=0$ in $H^{*}(G_{3,n};\mathbb{Z}_{2})$, it follows immediately that this is zero. Also,
\[w_{1}^{4}w_{2}^{8}w_{3}^{n-8}+w_{1}^{2}w_{2}^{9}w_{3}^{n-8}+w_{1}^{2}w_{2}^{12}w_{3}^{n-10}
+w_{1}^{2}w_{2}^{15}w_{3}^{n-12}\]
\[=w_{1}^{2}w_{3}(g_{8,n-9}+w_{2}^{3}w_{3}^{n-5}+w_{3}^{n-3})+w_{1}w_{2}^{2}(g_{10,n-10}+w_{1}^{3}w_{3}^{n-4}
+w_{1}^{2}w_{2}^{2}w_{3}^{n-5}+w_{3}^{n-3})\]
\[+w_{1}w_{2}^{3}(g_{12,n-12}+w_{1}^{4}w_{3}^{n-5}+w_{2}^{8}w_{3}^{n-9}+w_{1}w_{3}^{n-4})\]
\[=w_{1}^{2}w_{3}^{n-2}+w_{1}^{4}w_{2}^{2}w_{3}^{n-4}+w_{1}^{3}w_{2}^{4}w_{3}^{n-5}+w_{1}w_{2}^{2}w_{3}^{n-3}
+w_{1}^{5}w_{2}^{3}w_{3}^{n-5}+w_{1}w_{2}^{11}w_{3}^{n-9}\]
\[=w_{1}^{2}(g_{3,n-5}+w_{1}w_{2}w_{3}^{n-3})+w_{1}w_{2}(g_{10,n-9}+w_{1}^{2}w_{3}^{n-3})=0\]
and we have proved (a).

Likewise, from the upper expression for $w(\nu )$ one obtains that
\[w_{3n-2}(\nu )=w_{1}^{2}w_{2}w_{3}^{n-2}+w_{1}^{8}w_{2}w_{3}^{n-4}+w_{1}^{6}w_{2}^{2}w_{3}^{n-4}+w_{1}^{2}w_{2}^{4}w_{3}^{n-4}\]
\[+w_{1}^{4}w_{2}^{9}w_{3}^{n-8}+w_{1}^{2}w_{2}^{10}w_{3}^{n-8}+w_{1}^{2}w_{2}^{13}w_{3}^{n-10}.\]
Again, for $n=6$ we note that only first four summands appear and it is obvious that in this case $w_{3n-2}(\nu )=w_{2}w_{3n-4}(\nu )=0$ by (a). For $n>6$ (i.e., $n\geq 14$), we have that
\begin{eqnarray*}
w_{3n-2}(\nu ) & = & w_{2}(w_{3n-4}(\nu )+w_{1}^{2}w_{2}^{15}w_{3}^{n-12})=w_{1}^{2}w_{2}^{16}w_{3}^{n-12}\\
& = & w_{1}w_{2}^{4}(g_{12,n-12}+w_{1}^{4}w_{3}^{n-5}+w_{2}^{8}w_{3}^{n-9}+w_{1}w_{3}^{n-4})\\
& = & w_{1}^{5}w_{2}^{4}w_{3}^{n-5}+w_{1}w_{2}^{12}w_{3}^{n-9}+w_{1}^{2}w_{2}^{4}w_{3}^{n-4}\\
& = & w_{1}^{3}(g_{4,n-5}+w_{2}^{5}w_{3}^{n-5})
+w_{2}^{2}w_{3}(g_{10,n-10}+w_{1}^{3}w_{3}^{n-4}+w_{3}^{n-3})\\
& = & w_{1}^{3}w_{2}^{5}w_{3}^{n-5}+w_{1}^{3}w_{2}^{2}w_{3}^{n-3}+w_{2}^{2}w_{3}^{n-2}\\
& = & w_{1}^{2}(g_{5,n-5}+w_{2}^{4}w_{3}^{n-4})+w_{1}(g_{2,n-3}+w_{2}^{3}w_{3}^{n-3}+w_{3}^{n-1})+w_{2}^{2}w_{3}^{n-2}\\
& = & w_{1}^{2}w_{2}^{4}w_{3}^{n-4}
+w_{1}w_{2}^{3}w_{3}^{n-3}+w_{1}w_{3}^{n-1}+w_{2}^{2}w_{3}^{n-2}\\
& = & w_{1}g_{4,n-4}+g_{3,n-3}=0.
\end{eqnarray*} 
This proves (b).

\medskip

For (c), since $2^{r+1}-n-3\equiv 7\imod 8$, by equality (\ref{o1}) we have:
\[w_{2}(\nu )={2^{r+1}-n-3\choose 2}w_{1}^{2}+(2^{r+1}-n-3)w_{2}=w_{1}^{2}+w_{2}\]
and we are done.
\end{proof}

\begin{remark} The triviality of the classes $w_{3n-4}(\nu )$ and $w_{3n-2}(\nu )$ for the case $n=2^{r}-2$ and $r\geq 6$ is a consequence of a Massey's result (\cite[Theorem I]{Massey}) which states that $w_{i}(\nu)=0$ for $i>3n-\alpha(3n)=3n-r+1$. However, the calculation for the cases $r=3,4,5$ is not much simpler than for the arbitrary $r$ and therefore we have proved this fact for all $r\geq 3$.
\end{remark}

\begin{lemma}\label{l22} For the map $(Sq^{2}+w_{2}(\nu )):H^{3n-5}(G_{3,n};\mathbb{Z}_{2})\rightarrow H^{3n-3}(G_{3,n};\mathbb{Z}_{2})$, where $n\equiv 6\imod 8$, we have:
\begin{enumerate}
\item[] $(Sq^{2}+w_{2}(\nu ))(w_{1}w_{3}^{n-2})=w_{1}w_{2}w_{3}^{n-2}+w_{3}^{n-1}$;
\item[] $(Sq^{2}+w_{2}(\nu ))(w_{2}^{2}w_{3}^{n-3})=w_{2}^{3}w_{3}^{n-3}$.
\end{enumerate}
\end{lemma}
\begin{proof} We use Lemma \ref{l21} (c), formula (\ref{fsa2}) and Gr\"obner basis from Theorem \ref{t3} to calculate:
\begin{eqnarray*}
(Sq^{2}+w_{2}(\nu ))(w_{1}w_{3}^{n-2}) & = & (Sq^{2}+w_{1}^{2}+w_{2})(w_{1}w_{3}^{n-2})\\
& = & \bigl(\begin{smallmatrix} n-1\\ 2\end{smallmatrix}\bigr) w_{1}^{3}w_{3}^{n-2}+(n-2)w_{1}w_{2}w_{3}^{n-2}\\
& + & w_{1}^{3}w_{3}^{n-2}+w_{1}w_{2}w_{3}^{n-2}\\
& = & w_{1}^{3}w_{3}^{n-2}+w_{1}w_{2}w_{3}^{n-2}\\
& = & g_{0,n-2}+w_{3}^{n-1}+w_{1}w_{2}w_{3}^{n-2}=w_{1}w_{2}w_{3}^{n-2}+w_{3}^{n-1};
\end{eqnarray*}

\[(Sq^{2}+w_{1}^{2}+w_{2})(w_{2}^{2}w_{3}^{n-3})=\bigl(\begin{smallmatrix} n-1\\ 2\end{smallmatrix}\bigr) w_{1}^{2}w_{2}^{2}w_{3}^{n-3}+2 (n-3)w_{1}w_{2}w_{3}^{n-2}\]
\[+(n-1)w_{2}^{3}w_{3}^{n-3}+\bigl(\begin{smallmatrix} 2\\ 2\end{smallmatrix}\bigr) w_{3}^{n-1}+w_{1}^{2}w_{2}^{2}w_{3}^{n-3}+w_{2}^{3}w_{3}^{n-3}\]
\[=w_{3}^{n-1}+w_{1}^{2}w_{2}^{2}w_{3}^{n-3}=g_{2,n-3}+w_{2}^{3}w_{3}^{n-3}=w_{2}^{3}w_{3}^{n-3}\]
and the proof is completed.
\end{proof}

\begin{lemma}\label{l23} The map $Sq^{1}:H^{3n-2}(G_{3,n};\mathbb{Z}_{2})\rightarrow H^{3n-1}(G_{3,n};\mathbb{Z}_{2})$, where $n\equiv 6\imod 8$, is trivial.
\end{lemma}
\begin{proof} The set $\{w_{1}w_{3}^{n-1},w_{2}^{2}w_{3}^{n-2}\}$ is a vector space basis for $H^{3n-2}(G_{3,n};\mathbb{Z}_{2})$ (Corollary \ref{c2}). According to (\ref{fsa1}), we have:
\[Sq^{1}(w_{1}w_{3}^{n-1})=nw_{1}^{2}w_{3}^{n-1}=0;\]
\[Sq^{1}(w_{2}^{2}w_{3}^{n-2})=nw_{1}w_{2}^{2}w_{3}^{n-2}+2w_{2}w_{3}^{n-1}=0\]
and we are done.
\end{proof}

\begin{lemma}\label{l24} The map $(Sq^{2}+w_{2}(\nu )):H^{3n-3}(G_{3,n};\mathbb{Z}_{2})\rightarrow H^{3n-1}(G_{3,n};\mathbb{Z}_{2})$, where $n\equiv 6\imod 8$, is given by the equalities:
\begin{enumerate}
\item[] $(Sq^{2}+w_{2}(\nu ))(w_{1}w_{2}w_{3}^{n-2})=w_{2}w_{3}^{n-1}\neq 0$;
\item[] $(Sq^{2}+w_{2}(\nu ))(w_{2}^{3}w_{3}^{n-3})=0$;
\item[] $(Sq^{2}+w_{2}(\nu ))(w_{3}^{n-1})=w_{2}w_{3}^{n-1}$.
\end{enumerate}
\end{lemma}
\begin{proof} Again from Corollary \ref{c2}, we see that the set $\{w_{1}w_{2}w_{3}^{n-2},w_{2}^{3}w_{3}^{n-3},w_{3}^{n-1}\}$ is a vector space basis for $H^{3n-3}(G_{3,n};\mathbb{Z}_{2})$ and the class $w_{2}w_{3}^{n-1}$ is nontrivial in $H^{3n-1}(G_{3,n};\mathbb{Z}_{2})\cong \mathbb{Z}_{2}$.

We proceed to the calculation:
\[(Sq^{2}+w_{2}(\nu ))(w_{1}w_{2}w_{3}^{n-2})=(Sq^{2}+w_{1}^{2}+w_{2})(w_{1}w_{2}w_{3}^{n-2})\]
\[={n\choose 2}w_{1}^{3}w_{2}w_{3}^{n-2}+(n-1)w_{1}^{2}w_{3}^{n-1}+(n-1)w_{1}w_{2}^{2}w_{3}^{n-2}+w_{1}^{3}w_{2}w_{3}^{n-2}+w_{1}w_{2}^{2}w_{3}^{n-2}\]
\[=w_{1}^{2}w_{3}^{n-1}=g_{0,n-1}+w_{2}w_{3}^{n-1}=w_{2}w_{3}^{n-1},\]
by formula (\ref{fsa2}), Lemma \ref{l21} (c) and Gr\"obner basis from Theorem \ref{t3}. The proofs of the other two equalities are similar.
\end{proof}

\begin{lemma}\label{l25} The map $Sq^{1}:H^{3n-3}(G_{3,n};\mathbb{Z}_{2})\rightarrow H^{3n-2}(G_{3,n};\mathbb{Z}_{2})$, where $n\equiv 6\imod 8$, is given by the equalities:
\begin{enumerate}
\item[] $Sq^{1}(w_{1}w_{2}w_{3}^{n-2})=w_{1}w_{3}^{n-1}$;
\item[] $Sq^{1}(w_{2}^{3}w_{3}^{n-3})=w_{2}^{2}w_{3}^{n-2}$;
\item[] $Sq^{1}(w_{3}^{n-1})=w_{1}w_{3}^{n-1}$.
\end{enumerate}
\end{lemma}
\begin{proof} As we have already noticed in the proof of the previous lemma, the classes $w_{1}w_{2}w_{3}^{n-2}$, $w_{2}^{3}w_{3}^{n-3}$ and $w_{3}^{n-1}$ form a vector space basis for $H^{3n-3}(G_{3,n};\mathbb{Z}_{2})$. According to formula (\ref{fsa1}), we have:
\begin{eqnarray*}
Sq^{1}(w_{1}w_{2}w_{3}^{n-2}) &= & nw_{1}^{2}w_{2}w_{3}^{n-2}+w_{1}w_{3}^{n-1}=w_{1}w_{3}^{n-1};\\
Sq^{1}(w_{2}^{3}w_{3}^{n-3}) & = &nw_{1}w_{2}^{3}w_{3}^{n-3}+3w_{2}^{2}w_{3}^{n-2}=w_{2}^{2}w_{3}^{n-2};\\
Sq^{1}(w_{3}^{n-1}) & = & (n-1)w_{1}w_{3}^{n-1}=w_{1}w_{3}^{n-1},
\end{eqnarray*} 

and the lemma is proved.
\end{proof}

\begin{lemma}\label{l26} If $n\equiv 6\imod 8$, then in $H^{*}(G_{3,n};\mathbb{Z}_{2})$ we have
\[Sq^{2}(w_{1}w_{2}w_{3}^{n-2}+w_{3}^{n-1})=w_{2}w_{3}^{n-1}.\]
\end{lemma}
\begin{proof} Using (\ref{fsa2}), we calculate:
\[Sq^{2}(w_{1}w_{2}w_{3}^{n-2}+w_{3}^{n-1})=\bigl(\begin{smallmatrix} n\\ 2\end{smallmatrix}\bigr) w_{1}^{3}w_{2}w_{3}^{n-2}+(n-1)w_{1}^{2}w_{3}^{n-1}+(n-1)w_{1}w_{2}^{2}w_{3}^{n-2}\]
\[+\bigl(\begin{smallmatrix} n-1\\ 2\end{smallmatrix}\bigr) w_{1}^{2}w_{3}^{n-1}+(n-1)w_{2}w_{3}^{n-1}
=w_{1}^{3}w_{2}w_{3}^{n-2}+w_{1}^{2}w_{3}^{n-1}+w_{1}w_{2}^{2}w_{3}^{n-2}+w_{2}w_{3}^{n-1}\]
\[=w_{1}g_{1,n-2}+w_{2}w_{3}^{n-1}=w_{2}w_{3}^{n-1}\]
and we are done.
\end{proof}

Now, we are ready to prove Theorem \ref{theorem3}.

\medskip

\noindent{\bf Proof of Theorem \ref{theorem3}.} We shall prove that the classifying map for the stable normal bundle $\nu$ of $G_{3,n}$, $f_{\nu}:G_{3,n}\rightarrow BO$, can be lifted up to $BO(3n-5)$. The $3n$-MPT for the fibration $p:BO(3n-5)\rightarrow BO$ and the table of $k$-invariants of this tower are given below.
\[\bfig
 \morphism<900,0>[G_{3,n}`BO;f_{\nu}]
 \morphism(900,0)<1100,0>[BO`K_{3n-4}\times K_{3n-2};w_{3n-4}\times w_{3n-2}]
 \morphism(900,500)|r|<0,-500>[E_{1}`BO;q_{1}]
 \morphism(900,500)<1100,0>[E_{1}`K_{3n-3}\times K_{3n-2}\times K_{3n-1};k_{1}^{2}\times k_{2}^{2}\times k_{3}^{2}]
 \morphism(900,1000)|r|<0,-500>[E_{2}`E_{1};q_{2}]
 \morphism(900,1000)<1100,0>[E_{2}`K_{3n-2};k_{1}^{3}]
 \morphism(900,1500)|r|<0,-500>[E_{3}`E_{2};q_{3}]
 \morphism/-->/<900,500>[G_{3,n}`E_{1};g]
 \morphism/-->/<900,1000>[G_{3,n}`E_{2};h]
 \efig\]
\begin{table}[ht]
\label{eqtable}
\renewcommand\arraystretch{1.5}
\noindent\[
\begin{array}{|l|}
\hline
k_{1}^{2}: \quad (Sq^{2}+w_{2})w_{3n-4}=0\\
\hline
k_{2}^{2}: \quad (Sq^{2}+w_{1}^{2}+w_{2})Sq^{1}w_{3n-4}+Sq^{1}w_{3n-2}=0\\
\hline
k_{3}^{2}: \quad (Sq^{4}+w_{4})w_{3n-4}+Sq^{2}w_{3n-2}=0\\
\hline
k_{1}^{3}: \quad (Sq^{2}+w_{2})k_{1}^{2}+Sq^{1}k_{2}^{2}=0\\
\hline
\end{array}
\]
\end{table}

According to Lemma \ref{l21} (parts (a) and (b)), $f_{\nu}^{*}(w_{3n-4})=w_{3n-4}(\nu )=0$ and $f_{\nu}^{*}(w_{3n-2})=w_{3n-2}(\nu )=0$, so there is a lifting $g_{1}:G_{3,n}\rightarrow E_{1}$ of $f_{\nu}$.

\medskip

In order to make the next step (to lift $f_{\nu}$ up to $E_{2}$), we need to modify $g_{1}$  (if necessary) to a lifting $g$ such that $g^{*}(k_{1}^{2})=g^{*}(k_{2}^{2})=g^{*}(k_{3}^{2})=0$. By choosing a map $\alpha \times \beta :G_{3,n}\rightarrow K_{3n-5}\times K_{3n-3}=\Omega (K_{3n-4}\times K_{3n-2})$ (i.e., classes $\alpha \in H^{3n-5}(G_{3,n};\mathbb{Z}_{2})$ and $\beta \in H^{3n-3}(G_{3,n};\mathbb{Z}_{2})$), we get another lifting $g:G_{3,n}\rightarrow E_{1}$ as the composition:
\[\bfig
 \morphism<800,0>[G_{3,n}`G_{3,n}\times G_{3,n};\triangle]
 \morphism(800,0)<1200,0>[G_{3,n}\times G_{3,n}`K_{3n-5}\times K_{3n-3}\times E_{1};(\alpha \times \beta)\times g_{1}]
 \morphism(2000,0)<800,0>[K_{3n-5}\times K_{3n-3}\times E_{1}`E_{1},;\mu]
 \efig\]
where $\triangle$ is the diagonal mapping and $\mu :\Omega (K_{3n-4}\times K_{3n-2})\times E_{1}\rightarrow E_{1}$ is the action of the fibre in the principal fibration $q_{1}:E_{1}\rightarrow BO$. So, we are looking for classes $\alpha$ and $\beta$ such that $g^{*}(k_{1}^{2})=g^{*}(k_{2}^{2})=g^{*}(k_{3}^{2})=0$. By looking at the relations that produce the $k$-invariants $k_{1}^{2},k_{2}^{2}$ and $k_{3}^{2}$ we conclude that the following equalities hold (see \cite[p.\,95]{Gitler}):
\begin{enumerate}
\item[] $g^{*}(k_{1}^{2})=g_{1}^{*}(k_{1}^{2})+(Sq^{2}+w_{2}(\nu ))(\alpha )$;
\item[] $g^{*}(k_{2}^{2})=g_{1}^{*}(k_{2}^{2})+(Sq^{2}+w_{1}(\nu )^{2}+w_{2}(\nu ))Sq^{1}\alpha +Sq^{1}\beta$;
\item[] $g^{*}(k_{3}^{2})=g_{1}^{*}(k_{3}^{2})+(Sq^{4}+w_{4}(\nu ))(\alpha )+Sq^{2}\beta.$
\end{enumerate}
First we need to prove that $g_{1}^{*}(k_{1}^{2})$ is in the image of the map $(Sq^{2}+w_{2}(\nu )):H^{3n-5}(G_{3,n};\mathbb{Z}_{2})\rightarrow H^{3n-3}(G_{3,n};\mathbb{Z}_{2})$. Observe the relation $(Sq^{2}+w_{2})k_{1}^{2}+Sq^{1}k_{2}^{2}=0$ in $H^{*}(E_{1};\mathbb{Z}_{2})$ (which produces $k_{1}^{3}$). If we pull back this relation by $g_{1}^{*}$ to $H^{*}(G_{3,n};\mathbb{Z}_{2})$, we get: \[(Sq^{2}+w_{2}(\nu ))g_{1}^{*}(k_{1}^{2})=Sq^{1}g_{1}^{*}(k_{2}^{2}).\]
By Lemma \ref{l23}, $Sq^{1}g_{1}^{*}(k_{2}^{2})=0$ and we conclude that the class $g_{1}^{*}(k_{1}^{2})$ is in the kernel of the map $(Sq^{2}+w_{2}(\nu )):H^{3n-3}(G_{3,n};\mathbb{Z}_{2})\rightarrow H^{3n-1}(G_{3,n};\mathbb{Z}_{2})$. According to lemmas \ref{l22} and \ref{l24}, this kernel is contained in the image of the map $(Sq^{2}+w_{2}(\nu )):H^{3n-5}(G_{3,n};\mathbb{Z}_{2})\rightarrow H^{3n-3}(G_{3,n};\mathbb{Z}_{2})$ and so, we can find a class $\alpha \in H^{3n-5}(G_{3,n};\mathbb{Z}_{2})$ such that $g^{*}(k_{1}^{2})=0$.

By Corollary \ref{c2}, the classes $w_{1}w_{3}^{n-1}$ and $w_{2}^{2}w_{3}^{n-2}$ generate $H^{3n-2}(G_{3,n};\mathbb{Z}_{2})$ and now, from Lemma \ref{l25} it is obvious that there is a class $\beta \in H^{3n-3}(G_{3,n};\mathbb{Z}_{2})$ which produces a lifting $g$ such that $g^{*}(k_{1}^{2})=g^{*}(k_{2}^{2})=0$.

If, for these choices of $\alpha$ and $\beta$, $g^{*}(k_{3}^{2})\neq 0$, i.e., $g^{*}(k_{3}^{2})=w_{2}w_{3}^{n-1}$, we can add the class $w_{1}w_{2}w_{3}^{n-2}+w_{3}^{n-1}$ to $\beta$ and obtain a new $\beta\in H^{3n-3}(G_{3,n};\mathbb{Z}_{2})$. By Lemma \ref{l26}, now we have that $g^{*}(k_{3}^{2})=0$. Finally, since $w_{1}w_{2}w_{3}^{n-2}+w_{3}^{n-1}$ is in the kernel of $Sq^{1}:H^{3n-3}(G_{3,n};\mathbb{Z}_{2})\rightarrow H^{3n-2}(G_{3,n};\mathbb{Z}_{2})$ (Lemma \ref{l25}), we conclude that $g^{*}(k_{1}^{2})=g^{*}(k_{2}^{2})=g^{*}(k_{3}^{2})=0$.

Therefore, we can lift $f_{\nu}$ one more stage, i.e., there is a map $h:G_{3,n}\rightarrow E_{2}$ such that $q_{1}\circ q_{2}\circ h=q_{1}\circ g=f_{\nu}$.

\medskip

For the final step, we observe the relation that produces $k_{1}^{3}$ and note that the indeterminacy of $k_{1}^{3}$ is all of $H^{3n-2}(G_{3,n};\mathbb{Z}_{2})$ (by Lemma \ref{l25}). Hence, the lifting $h:G_{3,n}\rightarrow E_{2}$ can be chosen such that $h^{*}(k_{1}^{3})=0$. This completes the proof of the theorem.  \hfill $\Box$

\medskip

Our next task is to prove Theorem \ref{thm1}.

\begin{lemma}\label{ll2} Let $n\geq 3$ and $n\equiv 1\imod 8$. If $\nu$ is the stable normal bundle of $G_{3,n}$, then
\begin{enumerate}
\item[$\mathrm{(a)}$] $w_{i}(\nu )=0$ for $i\geq 3n-8$;
\item[$\mathrm{(b)}$] $w_{2}(\nu )=0$;
\item[$\mathrm{(c)}$] $w_{4}(\nu )=w_{2}^{2}$.
\end{enumerate}
\end{lemma}
\begin{proof} As above, let $r\geq 3$ be the integer such that $2^{r+1}<3n<2^{r+2}$.

If $n\geq 2^{r}$, then $n$ must be $\geq 2^{r}+1$. So we have that $2^{r+1}\leq 2n-2$. The top class in the expression (\ref{o1}), $(w_{1}^{2}w_{2}^{2}+w_{3}^{2})w_{3}^{2^{r+1}-n-3}$, is in degree $6+3\cdot (2^{r+1}-n-3)\leq 6+3\cdot (n-5)=3n-9$ and (a) follows in this case.

If $n<2^{r}$, then we actually have that $n<2^{r}-2$ (since $n\equiv 1\imod 8$), so formula (\ref{o2}) holds. The top class there is in degree $6+3\cdot (2^{r}-n-3)$ and, since $3n>2^{r+1}$, we have that $2^{r}<\frac{3}{2}n$, implying $6+3\cdot (2^{r}-n-3)<6+3\cdot \frac{n-6}{2}<6+3\cdot (n-6)=3n-12$. This proves (a).

Parts (b) and (c) we read off from formula (\ref{o1}) (using the fact that $2^{r+1}-n-3\equiv 4\imod 8$):
\[w_{2}(\nu )=\bigl(\begin{smallmatrix} 2^{r+1}-n-3\\ 2\end{smallmatrix}\bigr) w_{1}^{2}+(2^{r+1}-n-3)w_{2}=0,\]
\[w_{4}(\nu )=w_{1}^{4}+w_{2}^{2}+\bigl(\begin{smallmatrix} 2^{r+1}-n-3\\ 4 \end{smallmatrix}\bigr) w_{1}^{4}+\bigl(\begin{smallmatrix} 2^{r+1}-n-3\\ 3\end{smallmatrix}\bigr) \bigl(\begin{smallmatrix} 3\\ 1\end{smallmatrix}\bigr) w_{1}^{2}w_{2}\]
\[+\bigl(\begin{smallmatrix} 2^{r+1}-n-3\\ 2\end{smallmatrix}\bigr) \bigl(\begin{smallmatrix} 2\\ 1\end{smallmatrix}\bigr) w_{1}w_{3}+\bigl(\begin{smallmatrix} 2^{r+1}-n-3\choose 2\end{smallmatrix}\bigr) w_{2}^{2}=w_{2}^{2}\]
and the lemma follows.
\end{proof}

\begin{lemma}\label{ll3} Let $n\geq 3$, $n\equiv 1\imod 8$. For the map $Sq^{2}:H^{3n-6}(G_{3,n};\mathbb{Z}_{2})\rightarrow H^{3n-4}(G_{3,n};\mathbb{Z}_{2})$ we have:
\begin{enumerate}
\item[] $Sq^{2}(w_{1}^{2}w_{2}^{2}w_{3}^{n-4})=w_{1}^{2}w_{3}^{n-2}+w_{1}w_{2}^{2}w_{3}^{n-3}
    +w_{2}^{4}w_{3}^{n-4}+w_{2}w_{3}^{n-2}$;
\item[] $Sq^{2}(w_{1}w_{2}w_{3}^{n-3})=w_{1}^{2}w_{3}^{n-2}+w_{1}w_{2}^{2}w_{3}^{n-3}$;
\item[] $Sq^{2}(w_{3}^{n-2})=w_{1}^{2}w_{3}^{n-2}+w_{2}w_{3}^{n-2}$.
\end{enumerate}
\end{lemma}
\begin{proof} We use Gr\"obner basis $G$ to calculate:
\[Sq^{2}(w_{1}^{2}w_{2}^{2}w_{3}^{n-4})\!\!=\!\!\bigl(\begin{smallmatrix} n\\ 2\end{smallmatrix}\bigr) w_{1}^{4}w_{2}^{2}w_{3}^{n-4}+2(n-2)w_{1}^{3}w_{2}w_{3}^{n-3}+(n-2)w_{1}^{2}w_{2}^{3}w_{3}^{n-4}+\bigl(\begin{smallmatrix} 2\\ 2\end{smallmatrix}\bigr) w_{1}^{2}w_{3}^{n-2}\]
\[=w_{1}^{2}w_{2}^{3}w_{3}^{n-4}+w_{1}^{2}w_{3}^{n-2}=g_{3,n-4}+w_{1}w_{2}^{2}w_{3}^{n-3}
+w_{2}^{4}w_{3}^{n-4}+w_{2}w_{3}^{n-2}+w_{1}^{2}w_{3}^{n-2}.\]
Since $g_{m,l}=0$ in $H^{*}(G_{3,n};\mathbb{Z}_{2})$, we obtain the first equality. Also,
\[Sq^{2}(w_{1}w_{2}w_{3}^{n-3})=\bigl(\begin{smallmatrix} n-1\\ 2\end{smallmatrix}\bigr) w_{1}^{3}w_{2}w_{3}^{n-3}+(n-2)w_{1}^{2}w_{3}^{n-2}+(n-2)w_{1}w_{2}^{2}w_{3}^{n-3}\]
and using the congruence $n\equiv 1\imod 8$, we directly get the second equality. Similarly,
\[Sq^{2}(w_{3}^{n-2})=\bigl(\begin{smallmatrix} n-2\\ 2\end{smallmatrix}\bigr) w_{1}^{2}w_{3}^{n-2}+(n-2)w_{2}w_{3}^{n-2}=w_{1}^{2}w_{3}^{n-2}+w_{2}w_{3}^{n-2}\]
and we are done.
\end{proof}

\begin{lemma}\label{ll4} The map $Sq^{2}:H^{3n-4}(G_{3,n};\mathbb{Z}_{2})\rightarrow H^{3n-2}(G_{3,n};\mathbb{Z}_{2})$, where $n\geq 3$ and $n\equiv 1\imod 8$, is given by the following equalities:
\begin{enumerate}
\item[] $Sq^{2}(w_{1}^{2}w_{3}^{n-2})=w_{1}w_{3}^{n-1}+w_{2}^{2}w_{3}^{n-2}$,
\item[] $Sq^{2}(w_{1}w_{2}^{2}w_{3}^{n-3})=Sq^{2}(w_{2}^{4}w_{3}^{n-4})=Sq^{2}(w_{2}w_{3}^{n-2})=w_{1}w_{3}^{n-1}$.
\end{enumerate}
\end{lemma}
\begin{proof} According to Corollary \ref{c2}, the set $\{ w_{1}^{2}w_{3}^{n-2},w_{1}w_{2}^{2}w_{3}^{n-3},w_{2}^{4}w_{3}^{n-4},w_{2}w_{3}^{n-2}\}$ is a vector space basis for $H^{3n-4}(G_{3,n};\mathbb{Z}_{2})$. We proceed to the calculation.
\[Sq^{2}(w_{1}^{2}w_{3}^{n-2})=\bigl(\begin{smallmatrix} n\choose 2\end{smallmatrix}\bigr) w_{1}^{4}w_{3}^{n-2}+(n-2)w_{1}^{2}w_{2}w_{3}^{n-2}=w_{1}^{2}w_{2}w_{3}^{n-2}\]
\[ =g_{1,n-2}+w_{1}w_{3}^{n-1}+w_{2}^{2}w_{3}^{n-2}
=w_{1}w_{3}^{n-1}+w_{2}^{2}w_{3}^{n-2},\]
\[Sq^{2}(w_{1}w_{2}^{2}w_{3}^{n-3})=\bigl(\begin{smallmatrix} n\\ 2\end{smallmatrix}\bigr) w_{1}^{3}w_{2}^{2}w_{3}^{n-3}+2(n-2)w_{1}^{2}w_{2}w_{3}^{n-2}
+(n-1)w_{1}w_{2}^{3}w_{3}^{n-3}+\bigl(\begin{smallmatrix} 2\\ 2\end{smallmatrix}\bigr) w_{1}w_{3}^{n-1}\]
\[=w_{1}w_{3}^{n-1},\]
\[Sq^{2}(w_{2}^{4}w_{3}^{n-4})=\bigl(\begin{smallmatrix} n\\ 2\end{smallmatrix}\bigr) w_{1}^{2}w_{2}^{4}w_{3}^{n-4}+4\cdot (n-4)w_{1}w_{2}^{3}w_{3}^{n-3}
+nw_{2}^{5}w_{3}^{n-4}+\bigl(\begin{smallmatrix} 4\\ 2\end{smallmatrix}\bigr) w_{2}^{2}w_{3}^{n-2}\]\[=w_{2}^{5}w_{3}^{n-4}
=g_{5,n-4}+w_{1}w_{3}^{n-1}=w_{1}w_{3}^{n-1},\]
\[Sq^{2}(w_{2}w_{3}^{n-2})=\bigl(\begin{smallmatrix} n-1\\ 2\end{smallmatrix}\bigr) w_{1}^{2}w_{2}w_{3}^{n-2}+(n-2)w_{1}w_{3}^{n-1}
+(n-1)w_{2}^{2}w_{3}^{n-2}=w_{1}w_{3}^{n-1}.\]
\end{proof}

\begin{lemma}\label{ll5} The map $Sq^{1}:H^{3n-3}(G_{3,n};\mathbb{Z}_{2})\rightarrow H^{3n-2}(G_{3,n};\mathbb{Z}_{2})$, where $n\geq 3$ and $n\equiv 1\imod 8$, is given by the following equalities:
\begin{enumerate}
\item[] $Sq^{1}(w_{1}w_{2}w_{3}^{n-2})=w_{2}^{2}w_{3}^{n-2}$,
\item[] $Sq^{1}(w_{2}^{3}w_{3}^{n-3})=Sq^{1}(w_{3}^{n-1})=0$.
\end{enumerate}
\end{lemma}
\begin{proof} By Corollary \ref{c2}, the classes $w_{1}w_{2}w_{3}^{n-2}$, $w_{2}^{3}w_{3}^{n-3}$ and $w_{3}^{n-1}$ form an additive basis for $H^{3n-3}(G_{3,n};\mathbb{Z}_{2})$ . Using Gr\"obner basis $G$, we have:
\[Sq^{1}(w_{1}w_{2}w_{3}^{n-2})=nw_{1}^{2}w_{2}w_{3}^{n-2}+w_{1}w_{3}^{n-1}=g_{1,n-2}+w_{2}^{2}w_{3}^{n-2}=w_{2}^{2}w_{3}^{n-2},\]
\[Sq^{1}(w_{2}^{3}w_{3}^{n-3})=nw_{1}w_{2}^{3}w_{3}^{n-3}+3w_{2}^{2}w_{3}^{n-2}=w_{1}w_{2}^{3}w_{3}^{n-3}+w_{2}^{2}w_{3}^{n-2}=g_{3,n-3}=0,\]
\[Sq^{1}(w_{3}^{n-1})=(n-1)w_{1}w_{3}^{n-1}=0\]
and the lemma is proved.
\end{proof}

In the proof of the following lemma, we shall make use of the fact that for any cohomology class $u$ and any nonnegative integers $m$ and $k$,
\[Sq^{m}(u^{2^{k}})=
   \left\{
   \begin{array}{ll}
   (Sq^{\frac{m}{2^{k}}}u)^{2^{k}}, \; \quad 2^{k}\mid m  \\
   \quad \quad \quad 0, \; \quad 2^{k}\nmid m \\
   \end{array}
   \right..\]
The case $k=1$ is obtained from Cartan formula and the rest is easily proved by induction on $k$.

\begin{lemma}\label{ll6} For the class $w_{1}w_{2}^{4}w_{3}^{n-5}\in H^{3n-6}(G_{3,n};\mathbb{Z}_{2})$, where $n\geq 3$ and $n\equiv 1\imod 8$, we have the following:
\begin{enumerate}
\item[$\mathrm{(a)}$] $Sq^{2}Sq^{1}(w_{1}w_{2}^{4}w_{3}^{n-5})=w_{3}^{n-1}$,
\item[$\mathrm{(b)}$] $Sq^{2}(w_{1}w_{2}^{4}w_{3}^{n-5})=0$,
\item[$\mathrm{(c)}$] $(Sq^{4}+w_{2}^{2})(w_{1}w_{2}^{4}w_{3}^{n-5})=0$.
\end{enumerate}
\end{lemma}
\begin{proof} One has:
\[Sq^{1}(w_{1}w_{2}^{4}w_{3}^{n-5})\!\!=\!\!nw_{1}^{2}w_{2}^{4}w_{3}^{n-5}+4w_{1}w_{2}^{3}w_{3}^{n-4}\!
=\!w_{1}^{2}w_{2}^{4}w_{3}^{n-5}\!=\!g_{4,n-5}+w_{2}^{5}w_{3}^{n-5}\!=\!w_{2}^{5}w_{3}^{n-5}\]
and
\[Sq^{2}Sq^{1}(w_{1}w_{2}^{4}w_{3}^{n-5})=\bigl(\begin{smallmatrix} n\\ 2\end{smallmatrix}\bigr) w_{1}^{2}w_{2}^{5}w_{3}^{n-5}+5(n-5)w_{1}w_{2}^{4}w_{3}^{n-4}+nw_{2}^{6}w_{3}^{n-5}+\bigl(\begin{smallmatrix} 5\\ 2\end{smallmatrix}\bigr) w_{2}^{3}w_{3}^{n-3}\]
\[=w_{2}^{6}w_{3}^{n-5}=g_{6,n-5}+w_{3}^{n-1}=w_{3}^{n-1}.\]
This proves (a). Also,
\[Sq^{2}(w_{1}w_{2}^{4}w_{3}^{n-5})\!\!=\!\!\bigl(\begin{smallmatrix} n\\ 2\end{smallmatrix}\bigr) w_{1}^{3}w_{2}^{4}w_{3}^{n-5}+4(n-4)w_{1}^{2}w_{2}^{3}w_{3}^{n-4}+(n-1)w_{1}w_{2}^{5}w_{3}^{n-5}+\bigl(\begin{smallmatrix} 4\\ 2\end{smallmatrix}\bigr) w_{1}w_{2}^{2}w_{3}^{n-3}\]
and since $n\equiv 1\imod 8$, this is obviously equal to $0$. Finally, for (c) we use Cartan formula and we get:
\[(Sq^{4}+w_{2}^{2})(w_{1}w_{2}^{4}w_{3}^{n-5})
=w_{1}^{2}Sq^{3}(w_{2}^{4}w_{3}^{n-5})+w_{1}Sq^{4}(w_{2}^{4}w_{3}^{n-5})+w_{1}w_{2}^{6}w_{3}^{n-5}.\]
Now, since $n-5$ is divisible by $4$, $w_{2}^{4}w_{3}^{n-5}=\left(w_{2}w_{3}^{\frac{n-5}{4}}\right)^{4}$ and so $Sq^{3}(w_{2}^{4}w_{3}^{n-5})=0$ and
\[Sq^{4}(w_{2}^{4}w_{3}^{n-5})=\left(Sq^{1}\left(w_{2}w_{3}^{\frac{n-5}{4}}\right)\right)^{4}
=\left(\left(1+\frac{n-5}{4}\right)w_{1}w_{2}w_{3}^{\frac{n-5}{4}}+w_{3}^{\frac{n-5}{4}+1}\right)^{4}=w_{3}^{n-1},\]
where the latter equality holds because $\frac{n-5}{4}$ is an odd integer (since $n\equiv 1\imod 8$). We conclude that
\[(Sq^{4}+w_{2}^{2})(w_{1}w_{2}^{4}w_{3}^{n-5})=w_{1}w_{3}^{n-1}+w_{1}w_{2}^{6}w_{3}^{n-5}=w_{1}g_{6,n-5}=0\]
and the proof of the lemma is completed.
\end{proof}

\begin{lemma}\label{ll7} For the classes $w_{1}w_{2}^{2}w_{3}^{n-3}, w_{2}w_{3}^{n-2}\in H^{3n-4}(G_{3,n};\mathbb{Z}_{2})$, where $n\geq 3$ and $n\equiv 1\imod 8$, we have the following:
\begin{enumerate}
\item[$\mathrm{(a)}$] $Sq^{1}(w_{1}w_{2}^{2}w_{3}^{n-3})=w_{2}^{3}w_{3}^{n-3}+w_{3}^{n-1}$, \quad $Sq^{1}(w_{2}w_{3}^{n-2})=w_{3}^{n-1}$;
\item[$\mathrm{(b)}$] $Sq^{2}(w_{1}w_{2}^{2}w_{3}^{n-3}+w_{2}w_{3}^{n-2})=0$.
\end{enumerate}
\end{lemma}
\begin{proof} (a) We have:
\[Sq^{1}(w_{1}w_{2}^{2}w_{3}^{n-3})=nw_{1}^{2}w_{2}^{2}w_{3}^{n-3}+2w_{1}w_{2}w_{3}^{n-2}=w_{1}^{2}w_{2}^{2}w_{3}^{n-3}
=g_{2,n-3}+w_{2}^{3}w_{3}^{n-3}+w_{3}^{n-1}\]
\[=w_{2}^{3}w_{3}^{n-3}+w_{3}^{n-1},\]
\[Sq^{1}(w_{2}w_{3}^{n-2})=(n-1)w_{1}w_{2}w_{3}^{n-2}+w_{3}^{n-1}=w_{3}^{n-1}.\]

(b) Similarly,
\[Sq^{2}(w_{1}w_{2}^{2}w_{3}^{n-3}+w_{2}w_{3}^{n-2})={n\choose 2}w_{1}^{3}w_{2}^{2}w_{3}^{n-3}+2(n-2)w_{1}^{2}w_{2}w_{3}^{n-2}+(n-1)w_{1}w_{2}^{3}w_{3}^{n-3}\]
\[+{2\choose 2}w_{1}w_{3}^{n-1}+{n-1\choose 2}w_{1}^{2}w_{2}w_{3}^{n-2}+(n-2)w_{1}w_{3}^{n-1}+(n-1)w_{2}^{2}w_{3}^{n-2}=0\]
and we are done.
\end{proof}

\begin{lemma}\label{ll8} For the class $w_{1}w_{3}^{n-2}\in H^{3n-5}(G_{3,n};\mathbb{Z}_{2})$, where $n\geq 3$ and $n\equiv 1\imod 8$, we have that
\[Sq^{2}(w_{1}w_{3}^{n-2})=w_{1}w_{2}w_{3}^{n-2}.\]
\end{lemma}
\begin{proof} We simply calculate:
\[Sq^{2}(w_{1}w_{3}^{n-2})={n-1\choose 2}w_{1}^{3}w_{3}^{n-2}+(n-2)w_{1}w_{2}w_{3}^{n-2}=w_{1}w_{2}w_{3}^{n-2}\]
proving the lemma.
\end{proof}

\medskip

\noindent{\bf Proof of Theorem \ref{thm1}.} Since $n+3$ is even, Grassmannian $G_{3,n}$ is orientable (see \cite[p.\,179]{Oproiu}) and so, we can make the proof slightly easier by using the "orientable" version of Hirsch's theorem which states that a smooth orientable compact $m$-manifold $M^{m}$ immerses into $\mathbb{R}^{m+l}$ if and only if the classifying map $f_{\nu}:M^{m}\rightarrow BSO$ of the stable normal bundle $\nu$ of $M^{m}$ lifts up to $BSO(l)$.
\[\bfig
\morphism<600,0>[M^{m}`BSO;f_{\nu}]
\morphism(600,500)|r|<0,-500>[BSO(l)`BSO;p]
\morphism/-->/<600,500>[M^{m}`BSO(l);]
\efig\]
Hence, we need to lift $f_{\nu}:G_{3,n}\rightarrow BSO$ up to $BSO(3n-6)$. The $3n$-MPT for the fibration $p:BSO(3n-6)\rightarrow BSO$ is given in the following diagram.
\[\bfig
 \morphism<900,0>[G_{3,n}`BSO;f_{\nu}]
 \morphism(900,0)<1100,0>[BSO`K_{3n-5}\times K_{3n-3};w_{3n-5}\times w_{3n-3}]
 \morphism(900,500)|r|<0,-500>[E_{1}`BSO;q_{1}]
 \morphism(900,500)<1100,0>[E_{1}`K_{3n-4}\times K_{3n-3}\times K_{3n-2};k_{1}^{2}\times k_{2}^{2}\times k_{3}^{2}]
 \morphism(900,1000)|r|<0,-500>[E_{2}`E_{1};q_{2}]
 \morphism(900,1000)<1100,0>[E_{2}`K_{3n-3};k_{1}^{3}]
 \morphism(900,1500)|r|<0,-500>[E_{3}`E_{2};q_{3}]
 \morphism/-->/<900,500>[G_{3,n}`E_{1};g]
 \morphism/-->/<900,1000>[G_{3,n}`E_{2};h]
 \efig\]
The table of $k$-invariants is the following one:
\begin{table}[ht]
\label{eqtable}
\renewcommand\arraystretch{1.5}
\noindent\[
\begin{array}{|l|}
\hline
k_{1}^{2}: \quad (Sq^{2}+w_{2})w_{3n-5}=0\\
\hline
k_{2}^{2}: \quad (Sq^{2}+w_{2})Sq^{1}w_{3n-5}+Sq^{1}w_{3n-3}=0\\
\hline
k_{3}^{2}: \quad (Sq^{4}+w_{4})w_{3n-5}+Sq^{2}w_{3n-3}=0\\
\hline
k_{1}^{3}: \quad (Sq^{2}+w_{2})k_{1}^{2}+Sq^{1}k_{2}^{2}=0\\
\hline
\end{array}
\]
\end{table}

Since $\mathrm{dim}(G_{3,n})=3n$, $f_{\nu}$ lifts up to $BSO(3n-6)$ if and only if it lifts up to $E_{3}$.

According to Lemma \ref{ll2} (a), $f_{\nu}^{*}(w_{3n-5})=w_{3n-5}(\nu )=0$ and $f_{\nu}^{*}(w_{3n-3})=w_{3n-3}(\nu )=0$, so there is a lifting $g_{1}:G_{3,n}\rightarrow E_{1}$ of $f_{\nu}$.

Let us remark here that for every lifting $g:G_{3,n}\rightarrow E_{1}$ of $f_{\nu}$, one has
\begin{equation}\label{e1} Sq^{2}(g^{*}(k_{1}^{2}))=Sq^{1}(g^{*}(k_{2}^{2})).
\end{equation}
This is obtained by applying $g^{*}$ to the relation $(Sq^{2}+w_{2})k_{1}^{2}=Sq^{1}k_{2}^{2}$ in $H^{*}(E_{1};\mathbb{Z}_{2})$ (which produces the $k$-invariant $k_{1}^{3}$) and using Lemma \ref{ll2} (b).

\medskip

We have a lifting $g_{1}:G_{3,n}\rightarrow E_{1}$ and in order to make the next step (to lift $f_{\nu}$ up to $E_{2}$), we need to modify $g_{1}$ (if necessary) to a lifting $g$ such that $g^{*}(k_{1}^{2})=g^{*}(k_{2}^{2})=g^{*}(k_{3}^{2})=0$. By choosing a map $\alpha \times \beta :G_{3,n}\rightarrow K_{3n-6}\times K_{3n-4}=\Omega (K_{3n-5}\times K_{3n-3})$ (i.e., classes $\alpha \in H^{3n-6}(G_{3,n};\mathbb{Z}_{2})$ and $\beta \in H^{3n-4}(G_{3,n};\mathbb{Z}_{2})$), we get another lifting $g_{2}:G_{3,n}\rightarrow E_{1}$ (induced by $g_{1}$,$\alpha$ and $\beta$)  as the composition:
\[\bfig
 \morphism<800,0>[G_{3,n}`G_{3,n}\times G_{3,n};\triangle]
 \morphism(800,0)<1200,0>[G_{3,n}\times G_{3,n}`K_{3n-6}\times K_{3n-4}\times E_{1};(\alpha \times \beta)\times g_{1}]
 \morphism(2000,0)<800,0>[K_{3n-6}\times K_{3n-4}\times E_{1}`E_{1},;\mu]
 \efig\]
where $\triangle$ is the diagonal mapping and $\mu :\Omega (K_{3n-5}\times K_{3n-3})\times E_{1}\rightarrow E_{1}$ is the action of the fibre in the principal fibration $q_{1}:E_{1}\rightarrow BSO$. By looking at the relations that produce the $k$-invariants $k_{1}^{2},k_{2}^{2}$ and $k_{3}^{2}$ and using Lemma \ref{ll2} we conclude that the following equalities hold (see \cite[p.\,95]{Gitler}):
\begin{enumerate}
\item[] $g_{2}^{*}(k_{1}^{2})=g_{1}^{*}(k_{1}^{2})+(Sq^{2}+w_{2}(\nu ))(\alpha )=g_{1}^{*}(k_{1}^{2})+Sq^{2}\alpha$;
\item[] $g_{2}^{*}(k_{2}^{2})=g_{1}^{*}(k_{2}^{2})+(Sq^{2}+w_{2}(\nu ))Sq^{1}\alpha +Sq^{1}\beta=g_{1}^{*}(k_{2}^{2})+Sq^{2}Sq^{1}\alpha +Sq^{1}\beta$;
\item[] $g_{2}^{*}(k_{3}^{2})=g_{1}^{*}(k_{3}^{2})+(Sq^{4}+w_{4}(\nu ))(\alpha )+Sq^{2}\beta=g_{1}^{*}(k_{3}^{2})+(Sq^{4}+w_{2}^{2})(\alpha )+Sq^{2}\beta.$
\end{enumerate}
First we need to prove that $g_{1}^{*}(k_{1}^{2})$ is in the image of the map $Sq^{2}:H^{3n-6}(G_{3,n};\mathbb{Z}_{2})\rightarrow H^{3n-4}(G_{3,n};\mathbb{Z}_{2})$. Let us assume, to the contrary, that $g_{1}^{*}(k_{1}^{2})$ is not in this image. The classes $w_{1}^{2}w_{3}^{n-2}$, $w_{1}w_{2}^{2}w_{3}^{n-3}$, $w_{2}^{4}w_{3}^{n-4}$ and $w_{2}w_{3}^{n-2}$ form a vector space basis for $H^{3n-4}(G_{3,n};\mathbb{Z}_{2})$ (Corollary \ref{c2}) and from Lemma \ref{ll3} we conclude that the sum of all basis elements and the sum of any two basis elements are in the image of $Sq^{2}$. This means that $g_{1}^{*}(k_{1}^{2})$ is either a basis element or a sum of three distinct basis elements. Now, by looking at Lemma \ref{ll4}, we see that $Sq^{2}(g_{1}^{*}(k_{1}^{2}))\in \{ w_{1}w_{3}^{n-1},w_{1}w_{3}^{n-1}+w_{2}^{2}w_{3}^{n-2}\}$ and from formula (\ref{e1}) we have that $Sq^{2}(g_{1}^{*}(k_{1}^{2}))=Sq^{1}(g_{1}^{*}(k_{2}^{2}))$. But according to Lemma \ref{ll5}, $Sq^{1}(g_{1}^{*}(k_{2}^{2}))$ cannot belong to $\{ w_{1}w_{3}^{n-1},w_{1}w_{3}^{n-1}+w_{2}^{2}w_{3}^{n-2}\}$. This contradiction proves that we can find a class $\alpha \in H^{3n-6}(G_{3,n};\mathbb{Z}_{2})$ such that $Sq^{2}\alpha =g_{1}^{*}(k_{1}^{2})$.

The set $\{ w_{1}w_{3}^{n-1},w_{2}^{2}w_{3}^{n-2}\}$ is a vector space basis for $H^{3n-2}(G_{3,n};\mathbb{Z}_{2})$ (Corollary \ref{c2}) and by Lemma \ref{ll4}, there is a class $\beta \in H^{3n-4}(G_{3,n};\mathbb{Z}_{2})$ such that $Sq^{2}\beta =g_{1}^{*}(k_{3}^{2})+(Sq^{4}+w_{2}^{2})(\alpha )$ and so we have a lifting $g_{2}:G_{3,n}\rightarrow E_{1}$ (induced by $g_{1}$ and these classes $\alpha$ and $\beta$) such that $g_{2}^{*}(k_{1}^{2})=g_{2}^{*}(k_{3}^{2})=0$.

There is one more obstruction to lifting $f_{\nu}$ up to $E_{2}$: $g_{2}^{*}(k_{2}^{2})\in H^{3n-3}(G_{3,n};\mathbb{Z}_{2})$. Since $g_{2}^{*}(k_{1}^{2})=0$, by equality (\ref{e1}), we have that $Sq^{1}(g_{2}^{*}(k_{2}^{2}))=0$ and according to Lemma \ref{ll5}, $g_{2}^{*}(k_{2}^{2})$ must be in the subgroup of $H^{3n-3}(G_{3,n};\mathbb{Z}_{2})$ generated by $w_{2}^{3}w_{3}^{n-3}$ and $w_{3}^{n-1}$. Observe the classes $\alpha':=w_{1}w_{2}^{4}w_{3}^{n-5}\in H^{3n-6}(G_{3,n};\mathbb{Z}_{2})$ and $\beta':=w_{1}w_{2}^{2}w_{3}^{n-3}+w_{2}w_{3}^{n-2}\in H^{3n-4}(G_{3,n};\mathbb{Z}_{2})$. By Lemma \ref{ll6} (a), $Sq^{2}Sq^{1}\alpha'=w_{3}^{n-1}$ and according to Lemma \ref{ll7} (a), $Sq^{1}\beta'=w_{2}^{3}w_{3}^{n-3}$. This means that we can choose the coefficients $a,b\in \{ 0,1\}$ such that $Sq^{2}Sq^{1}(a\alpha')+Sq^{1}(b\beta')=g_{2}^{*}(k_{2}^{2})$. Finally, from Lemma \ref{ll6}, parts (b) and (c), and Lemma \ref{ll7} (b), we conclude that for the lifting $g:G_{3,n}\rightarrow E_{1}$ induced by $g_{2}$ and the classes $a\alpha'$ and $b\beta'$, all obstructions vanish, i.e., $g^{*}(k_{1}^{2})=g^{*}(k_{2}^{2})=g^{*}(k_{3}^{2})=0$.

Therefore, the lifting $g$ lifts up to $E_{2}$, i.e., there is a map $h:G_{3,n}\rightarrow E_{2}$ such that $q_{1}\circ q_{2}\circ h=q_{1}\circ g=f_{\nu}$.

\medskip

For the final step, we observe that the set $\{w_{1}w_{2}w_{3}^{n-2},w_{2}^{3}w_{3}^{n-3},w_{3}^{n-1}\}$ is a vector space basis for $H^{3n-3}(G_{3,n};\mathbb{Z}_{2})$ (Corollary \ref{c2}). By looking at the relation that produces the $k$-invariant $k_{1}^{3}$ and according to Lemma \ref{ll7} (a) and Lemma \ref{ll8}, one sees that the indeterminacy of $k_{1}^{3}$ is all of $H^{3n-3}(G_{3,n};\mathbb{Z}_{2})$. Hence, the lifting $h:G_{3,n}\rightarrow E_{2}$ can be chosen such that $h^{*}(k_{1}^{3})=0$. This completes the proof of the theorem.  \qed

\medskip

We are left to prove Theorem \ref{theorem4}. Several lemmas will be helpful.

\begin{lemma}\label{l27} Let $n\geq 3$ and $n\equiv 2\imod 8$. If $\nu$ is the stable normal bundle of $G_{3,n}$, then
\begin{enumerate}
\item[$\mathrm{(a)}$] $w_{i}(\nu )=0$ for $i\geq 3n-14$;
\item[$\mathrm{(b)}$] $w_{1}(\nu )=w_{1}$;
\item[$\mathrm{(c)}$] $w_{2}(\nu )=w_{1}^{2}+w_{2}$;
\item[$\mathrm{(d)}$] $w_{3}(\nu )=w_{1}^{3}+w_{3}$;
\item[$\mathrm{(e)}$] $w_{4}(\nu )=w_{1}^{4}+w_{1}^{2}w_{2}$.
\end{enumerate}
\end{lemma}
\begin{proof} As before, let $r\geq 3$ be the integer such that $2^{r+1}<3n<2^{r+2}$.

If $n\geq 2^{r}$, i.e., $n\geq 2^{r}+2$, then $2^{r+1}\leq 2n-4$. The top class in the expression (\ref{o1}) is in degree $6+3\cdot (2^{r+1}-n-3)\leq 6+3\cdot (n-7)=3n-15$ and (a) follows in this case.

If $n<2^{r}$ then $n$ must be $<2^{r}-2$ (since $n\equiv 2\imod 8$), so formula (\ref{o2}) holds. The top class there is in degree $6+3\cdot (2^{r}-n-3)$ and, since $3n\geq 2^{r+1}+1$, we have that $2^{r}\leq \frac{3n-1}{2}$, implying $6+3\cdot (2^{r}-n-3)\leq 6+3\cdot \frac{n-7}{2}<6+3\cdot (n-7)=3n-15$. This proves (a).

Parts (b), (c), (d) and (e) we read off from formula (\ref{o1}) (using the fact that $2^{r+1}-n-3\equiv 3\imod 8$):
\[w_{1}(\nu )=(2^{r+1}-n-3)w_{1}=w_{1},\]
\[w_{2}(\nu )=\bigl(\begin{smallmatrix} 2^{r+1}-n-3\\ 2\end{smallmatrix}\bigr) w_{1}^{2}+(2^{r+1}-n-3)w_{2}=w_{1}^{2}+w_{2},\]
\[w_{3}(\nu )=\bigl(\begin{smallmatrix} 2^{r+1}-n-3\\ 3\end{smallmatrix}\bigr) w_{1}^{3}+\bigl(\begin{smallmatrix} 2^{r+1}-n-3\\ 2\end{smallmatrix}\bigr) \bigl(\begin{smallmatrix} 2\\ 1\end{smallmatrix}\bigr) w_{1}w_{2}+(2^{r+1}-n-3)w_{3}=w_{1}^{3}+w_{3},\]
\[w_{4}(\nu )=w_{1}^{4}+w_{2}^{2}+\bigl(\begin{smallmatrix} 2^{r+1}-n-3\\ 4\end{smallmatrix}\bigr) w_{1}^{4}+\bigl(\begin{smallmatrix} 2^{r+1}-n-3\\ 3\end{smallmatrix}\bigr) \bigl(\begin{smallmatrix} 3\\ 1\end{smallmatrix}\bigr) w_{1}^{2}w_{2}\]
\[+\bigl(\begin{smallmatrix} 2^{r+1}-n-3\\ 2\end{smallmatrix}\bigr) \bigl(\begin{smallmatrix} 2\\ 1\end{smallmatrix}\bigr) w_{1}w_{3}+\bigl(\begin{smallmatrix} 2^{r+1}-n-3\\ 2\end{smallmatrix}\bigr) w_{2}^{2}=w_{1}^{4}+w_{1}^{2}w_{2}\]
and the lemma follows.
\end{proof}

\begin{lemma}\label{l28} Let $n$ be an integer $\geq 3$ such that $n\equiv 2\imod 8$. Then, for the map $F_{1}:=(Sq^{2}+w_{1}(\nu )^{2}+w_{2}(\nu ))Sq^{1}:H^{3n-7}(G_{3,n};\mathbb{Z}_{2})\rightarrow H^{3n-4}(G_{3,n};\mathbb{Z}_{2})$ we have
\begin{enumerate}
\item[] $F_{1}(w_{1}^{3}w_{2}w_{3}^{n-4})=F_{1}(w_{1}^{2}w_{3}^{n-3})=w_{1}^{2}w_{3}^{n-2}+w_{1}w_{2}^{2}w_{3}^{n-3}$;
\item[] $F_{1}(w_{1}w_{2}^{5}w_{3}^{n-6})=F_{1}(w_{2}^{4}w_{3}^{n-5})=w_{2}^{4}w_{3}^{n-4}$;
\item[] $F_{1}(w_{2}^{7}w_{3}^{n-7})=F_{1}(w_{2}w_{3}^{n-3})=w_{2}w_{3}^{n-2}$.
\end{enumerate}
\end{lemma}
\begin{proof} By Lemma \ref{l27}, $F_{1}=(Sq^{2}+w_{2})Sq^{1}$. According to (\ref{fsa1}), $Sq^{1}(w_{1}^{3}w_{2}w_{3}^{n-4})=nw_{1}^{4}w_{2}w_{3}^{n-4}+w_{1}^{3}w_{3}^{n-3}=w_{1}^{3}w_{3}^{n-3}$ and $Sq^{1}(w_{1}^{2}w_{3}^{n-3})=(n-1)w_{1}^{3}w_{3}^{n-3}=w_{1}^{3}w_{3}^{n-3}$ too. So,
\[F_{1}(w_{1}^{3}w_{2}w_{3}^{n-4})=F_{1}(w_{1}^{2}w_{3}^{n-3})=(Sq^{2}+w_{2})(w_{1}^{3}w_{3}^{n-3})\]
\[=Sq^{2}(w_{1}^{3}w_{3}^{n-3})+w_{1}^{3}w_{2}w_{3}^{n-3}={n\choose 2}w_{1}^{5}w_{3}^{n-3}+(n-3)w_{1}^{3}w_{2}w_{3}^{n-3}+w_{1}^{3}w_{2}w_{3}^{n-3}\]
\[=w_{1}^{5}w_{3}^{n-3}=w_{1}(g_{0,n-3}+w_{1}^{2}w_{2}w_{3}^{n-3}+w_{2}^{2}w_{3}^{n-3})
=w_{1}^{3}w_{2}w_{3}^{n-3}+w_{1}w_{2}^{2}w_{3}^{n-3}\]
\[=g_{1,n-3}+w_{1}^{2}w_{3}^{n-2}+w_{1}w_{2}^{2}w_{3}^{n-3}=w_{1}^{2}w_{3}^{n-2}+w_{1}w_{2}^{2}w_{3}^{n-3},\]
by (\ref{fsa2}) and Gr\"obner basis from Theorem \ref{t3}. The remaining equalities are proved similarly.
\end{proof}

\begin{lemma}\label{l29} Let $n$ be an integer $\geq 3$ such that $n\equiv 2\imod 8$. Then, for the map $D:=(Sq^{2}+w_{1}(\nu )^{2}+w_{2}(\nu )):H^{3n-4}(G_{3,n};\mathbb{Z}_{2})\rightarrow H^{3n-2}(G_{3,n};\mathbb{Z}_{2})$ we have
\begin{enumerate}
\item[] $D(w_{1}^{2}w_{3}^{n-2})=D(w_{1}w_{2}^{2}w_{3}^{n-3})=w_{2}^{2}w_{3}^{n-2}\neq 0$;
\item[] $D(w_{2}^{4}w_{3}^{n-4})=D(w_{2}w_{3}^{n-2})=0$.
\end{enumerate}
\end{lemma}
\begin{proof} As for the previous lemma, we shall prove the first two equalities and omit the proof of the other two (since it is analogous). We have that $D=Sq^{2}+w_{2}$ (Lemma \ref{l27}) and so,
\[D(w_{1}^{2}w_{3}^{n-2})=Sq^{2}(w_{1}^{2}w_{3}^{n-2})+w_{1}^{2}w_{2}w_{3}^{n-2}={n\choose 2}w_{1}^{4}w_{3}^{n-2}+(n-2)w_{1}^{2}w_{2}w_{3}^{n-2}\]
\[+w_{1}^{2}w_{2}w_{3}^{n-2}=w_{1}^{4}w_{3}^{n-2}+w_{1}^{2}w_{2}w_{3}^{n-2}=w_{1}g_{0,n-2}+g_{1,n-2}+w_{2}^{2}w_{3}^{n-2}=w_{2}^{2}w_{3}^{n-2};\]

\[D(w_{1}w_{2}^{2}w_{3}^{n-3})=Sq^{2}(w_{1}w_{2}^{2}w_{3}^{n-3})+w_{1}w_{2}^{3}w_{3}^{n-3}\]
\[={n\choose 2}w_{1}^{3}w_{2}^{2}w_{3}^{n-3}+2(n-2)w_{1}^{2}w_{2}w_{3}^{n-2}+(n-1)w_{1}w_{2}^{3}w_{3}^{n-3}+{2\choose 2}w_{1}w_{3}^{n-1}\]
\[+w_{1}w_{2}^{3}w_{3}^{n-3}=w_{1}^{3}w_{2}^{2}w_{3}^{n-3}+w_{1}w_{3}^{n-1}=w_{1}g_{2,n-3}+w_{1}w_{2}^{3}w_{3}^{n-3}=w_{1}w_{2}^{3}w_{3}^{n-3}\]
\[=g_{3,n-3}+w_{2}^{2}w_{3}^{n-2}=w_{2}^{2}w_{3}^{n-2}.\]
The fact $w_{2}^{2}w_{3}^{n-2}\neq 0$ is a direct consequence of Corollary \ref{c2}.
\end{proof}

\begin{lemma}\label{l30} Let $n$ be an integer $\geq 3$ such that $n\equiv 2\imod 8$ and let $F_{2}$ be the map $[(Sq^{4}+w_{2}(\nu )^{2}+w_{4}(\nu ))Sq^{1}+(w_{1}(\nu )w_{2}(\nu )+w_{3}(\nu ))Sq^{2}+(w_{1}(\nu )^{2}+w_{2}(\nu ))Sq^{3}]:H^{3n-7}(G_{3,n};\mathbb{Z}_{2})\rightarrow H^{3n-2}(G_{3,n};\mathbb{Z}_{2})$. Then
\begin{enumerate}
\item[] $F_{2}(w_{1}^{3}w_{2}w_{3}^{n-4}+w_{1}^{2}w_{3}^{n-3})=w_{1}w_{3}^{n-1}+w_{2}^{2}w_{3}^{n-2}$;
\item[] $F_{2}(w_{1}w_{2}^{5}w_{3}^{n-6}+w_{2}^{4}w_{3}^{n-5})=0$;
\item[] $F_{2}(w_{2}^{7}w_{3}^{n-7}+w_{2}w_{3}^{n-3})=w_{2}^{2}w_{3}^{n-2}$.
\end{enumerate}
\end{lemma}
\begin{proof} By Lemma \ref{l27}, $F_{2}=(Sq^{4}+w_{1}^{2}w_{2}+w_{2}^{2})Sq^{1}+(w_{1}w_{2}+w_{3})Sq^{2}+w_{2}Sq^{3}$. As we have already shown in the proof of Lemma \ref{l28}, $Sq^{1}(w_{1}^{3}w_{2}w_{3}^{n-4}+w_{1}^{2}w_{3}^{n-3})=0$.

For $Sq^{2}$, according to (\ref{fsa2}) we have:
\[Sq^{2}(w_{1}^{3}w_{2}w_{3}^{n-4}+w_{1}^{2}w_{3}^{n-3})={n\choose 2}w_{1}^{5}w_{2}w_{3}^{n-4}+(n-1)w_{1}^{4}w_{3}^{n-3}\]
\[+(n-3)w_{1}^{3}w_{2}^{2}w_{3}^{n-4}+{n-1\choose 2}w_{1}^{4}w_{3}^{n-3}+(n-3)w_{1}^{2}w_{2}w_{3}^{n-3}=w_{1}^{5}w_{2}w_{3}^{n-4}\]
\[+w_{1}^{4}w_{3}^{n-3}+w_{1}^{3}w_{2}^{2}w_{3}^{n-4}+w_{1}^{2}w_{2}w_{3}^{n-3}
=w_{1}(g_{1,n-4}+w_{2}^{3}w_{3}^{n-4}+w_{3}^{n-2})+w_{1}^{2}w_{2}w_{3}^{n-3}\]
\[=w_{1}^{2}w_{2}w_{3}^{n-3}+w_{1}w_{2}^{3}w_{3}^{n-4}+w_{1}w_{3}^{n-2}.\]

Since $Sq^{3}=Sq^{1}Sq^{2}$, we use the previous equality and (\ref{fsa1}) to calculate:
\[Sq^{3}(w_{1}^{3}w_{2}w_{3}^{n-4}+w_{1}^{2}w_{3}^{n-3})=Sq^{1}(w_{1}^{2}w_{2}w_{3}^{n-3}+w_{1}w_{2}^{3}w_{3}^{n-4}+w_{1}w_{3}^{n-2})\]
\[=nw_{1}^{3}w_{2}w_{3}^{n-3}+w_{1}^{2}w_{3}^{n-2}+nw_{1}^{2}w_{2}^{3}w_{3}^{n-4}+3w_{1}w_{2}^{2}w_{3}^{n-3}
+(n-1)w_{1}^{2}w_{3}^{n-2}=w_{1}w_{2}^{2}w_{3}^{n-3}.\]

By collecting all these facts, we obtain:
\[F_{2}(w_{1}^{3}w_{2}w_{3}^{n-4}+w_{1}^{2}w_{3}^{n-3})\]
\[=(w_{1}w_{2}+w_{3})(w_{1}^{2}w_{2}w_{3}^{n-3}+w_{1}w_{2}^{3}w_{3}^{n-4}+w_{1}w_{3}^{n-2})+w_{2}w_{1}w_{2}^{2}w_{3}^{n-3}\]
\[=w_{1}^{3}w_{2}^{2}w_{3}^{n-3}+w_{1}^{2}w_{2}^{4}w_{3}^{n-4}+w_{1}w_{3}^{n-1}=w_{1}(g_{2,n-3}+w_{2}^{3}w_{3}^{n-3}+w_{3}^{n-1})+w_{1}g_{4,n-4}\]
\[=w_{1}w_{2}^{3}w_{3}^{n-3}+w_{1}w_{3}^{n-1}=g_{3,n-3}+w_{2}^{2}w_{3}^{n-2}+w_{1}w_{3}^{n-1}
=w_{1}w_{3}^{n-1}+w_{2}^{2}w_{3}^{n-2}.\]

The proofs of the remaining equalities are similar.
\end{proof}

\begin{lemma}\label{l31} Let $n$ be an integer $\geq 3$, $n\equiv 2\imod 8$ and let $F_{3}$ be the map $[(Sq^{4}+w_{2}(\nu )^{2}+w_{4}(\nu ))Sq^{2}+(w_{1}(\nu )w_{2}(\nu )+w_{3}(\nu ))Sq^{3}]:H^{3n-7}(G_{3,n};\mathbb{Z}_{2})\rightarrow H^{3n-1}(G_{3,n};\mathbb{Z}_{2})$. Then
\[F_{3}(w_{1}w_{2}^{5}w_{3}^{n-6}+w_{2}^{4}w_{3}^{n-5})=w_{2}w_{3}^{n-1}\neq 0.\]
\end{lemma}
\begin{proof} Again by Lemma \ref{l27}, $F_{3}=(Sq^{4}+w_{1}^{2}w_{2}+w_{2}^{2})Sq^{2}+(w_{1}w_{2}+w_{3})Sq^{3}$. In the same manner as in the previous proofs, one can show that:
\[Sq^{2}(w_{1}w_{2}^{5}w_{3}^{n-6}+w_{2}^{4}w_{3}^{n-5})=w_{2}^{5}w_{3}^{n-5}+w_{1}w_{3}^{n-2};\]
\[Sq^{3}(w_{1}w_{2}^{5}w_{3}^{n-6}+w_{2}^{4}w_{3}^{n-5})=w_{1}^{2}w_{3}^{n-2}+w_{2}^{4}w_{3}^{n-4}.\]
Let us now calculate $Sq^{4}(Sq^{2}(w_{1}w_{2}^{5}w_{3}^{n-6}+w_{2}^{4}w_{3}^{n-5}))$. By formulae of Cartan and Wu (and Gr\"obner basis $G$):
\[Sq^{4}(w_{2}^{5}w_{3}^{n-5}+w_{1}w_{3}^{n-2})=(w_{1}^{4}w_{2}^{5}+w_{2}w_{3}^{4})w_{3}^{n-5}
+w_{2}^{6}w_{2}w_{3}^{n-5}+(w_{1}w_{2}^{5}+w_{2}^{4}w_{3})w_{3}^{n-4}\]
\[+w_{2}^{5}w_{1}^{4}w_{3}^{n-5}
=w_{2}w_{3}^{n-1}+w_{2}^{7}w_{3}^{n-5}+w_{1}w_{2}^{5}w_{3}^{n-4}+w_{2}^{4}w_{3}^{n-3}=w_{2}g_{6,n-5}\]
\[+w_{1}(g_{5,n-4}+w_{1}w_{3}^{n-1})+g_{4,n-3}+w_{2}w_{3}^{n-1}=w_{1}^{2}w_{3}^{n-1}+w_{2}w_{3}^{n-1}=g_{0,n-1}=0.\]

Finally, we have that
\[F_{3}(w_{1}w_{2}^{5}w_{3}^{n-6}+w_{2}^{4}w_{3}^{n-5})\]
\[=(w_{1}^{2}w_{2}+w_{2}^{2})(w_{2}^{5}w_{3}^{n-5}+w_{1}w_{3}^{n-2})+(w_{1}w_{2}+w_{3})(w_{1}^{2}w_{3}^{n-2}+w_{2}^{4}w_{3}^{n-4})\]
\[=w_{1}^{2}w_{2}^{6}w_{3}^{n-5}
+w_{2}^{7}w_{3}^{n-5}+w_{1}w_{2}^{2}w_{3}^{n-2}+w_{1}w_{2}^{5}w_{3}^{n-4}+w_{1}^{2}w_{3}^{n-1}+w_{2}^{4}w_{3}^{n-3}\]
\[=w_{1}w_{2}g_{5,n-5}+w_{2}(g_{6,n-5}+w_{3}^{n-1})
+g_{2,n-2}+g_{0,n-1}+g_{4,n-3}=w_{2}w_{3}^{n-1}\]
and from Corollary \ref{c2}, we directly deduce that $w_{2}w_{3}^{n-1}\neq 0$.
\end{proof}

\begin{lemma}\label{l32} Let $n$ be an integer $\geq 3$ such that $n\equiv 2\imod 8$. Then, for the map $H:=(Sq^{2}+w_{1}(\nu )^{2}+w_{2}(\nu )):H^{3n-5}(G_{3,n};\mathbb{Z}_{2})\rightarrow H^{3n-3}(G_{3,n};\mathbb{Z}_{2})$ we have
\begin{enumerate}
\item[] $H(w_{1}^{2}w_{2}w_{3}^{n-3})=w_{2}^{3}w_{3}^{n-3}+w_{3}^{n-1}$;
\item[] $H(w_{1}w_{3}^{n-2})=w_{1}w_{2}w_{3}^{n-2}$;
\item[] $H(w_{2}^{2}w_{3}^{n-3})=w_{3}^{n-1}$.
\end{enumerate}
\end{lemma}
\begin{proof} We prove the first equality only. $H=Sq^{2}+w_{2}$ (Lemma \ref{l27}) and so:
\[H(w_{1}^{2}w_{2}w_{3}^{n-3})=Sq^{2}(w_{1}^{2}w_{2}w_{3}^{n-3})+w_{1}^{2}w_{2}^{2}w_{3}^{n-3}={n\choose 2}w_{1}^{4}w_{2}w_{3}^{n-3}+(n-1)w_{1}^{3}w_{3}^{n-2}\]
\[+(n-2)w_{1}^{2}w_{2}^{2}w_{3}^{n-3}+w_{1}^{2}w_{2}^{2}w_{3}^{n-3}=w_{1}^{4}w_{2}w_{3}^{n-3}
+w_{1}^{3}w_{3}^{n-2}+w_{1}^{2}w_{2}^{2}w_{3}^{n-3}\]
\[=w_{1}g_{1,n-3}+g_{2,n-3}+w_{2}^{3}w_{3}^{n-3}+w_{3}^{n-1}=w_{2}^{3}w_{3}^{n-3}+w_{3}^{n-1},\]
by (\ref{fsa2}).
\end{proof}

\begin{lemma}\label{l33} Let $n\geq 3$, $n\equiv 2\imod 8$. In $H^{*}(G_{3,n};\mathbb{Z}_{2})$ the folowing equalities hold:
\begin{enumerate}
\item[] $Sq^{1}(w_{2}^{3}w_{3}^{n-3})=w_{2}^{2}w_{3}^{n-2}$;
\item[] $Sq^{1}(w_{3}^{n-1})=w_{1}w_{3}^{n-1}$.
\end{enumerate}
\end{lemma}
\begin{proof} By (\ref{fsa1}) we have:
\[Sq^{1}(w_{2}^{3}w_{3}^{n-3})=nw_{1}w_{2}^{3}w_{3}^{n-3}+3w_{2}^{2}w_{3}^{n-2}=w_{2}^{2}w_{3}^{n-2};\] 
\[Sq^{1}(w_{3}^{n-1})=(n-1)w_{1}w_{3}^{n-1}=w_{1}w_{3}^{n-1}\]
and we are done.
\end{proof}

Finally, we come to the proof of Theorem \ref{theorem4}.

\medskip

\noindent{\bf Proof of Theorem \ref{theorem4}.} As in the proofs of previous theorems, we are going to lift the classifying map $f_{\nu}:G_{3,n}\rightarrow BO$ up to $BO(3n-7)$. Since $n\equiv 2\imod 8$, we have that $3n\equiv 14\imod 8$, so $3n-7\equiv 7\imod 8$. This means that $3n$-MPT for the fibration $p:BO(3n-7)\rightarrow BO$ is of the following form.
\[\bfig
 \morphism<900,0>[G_{3,n}`BO;f_{\nu}]
 \morphism(900,0)<1100,0>[BO`K_{3n-6};w_{3n-6}]
 \morphism(900,500)|r|<0,-500>[E_{1}`BO;q_{1}]
 \morphism(900,500)<1100,0>[E_{1}`K_{3n-4}\times K_{3n-2}\times K_{3n-1};k_{1}^{2}\times k_{2}^{2}\times k_{3}^{2}]
 \morphism(900,1000)|r|<0,-500>[E_{2}`E_{1};q_{2}]
 \morphism(900,1000)<1100,0>[E_{2}`K_{3n-3}\times K_{3n-2};k_{1}^{3}\times k_{2}^{3}]
 \morphism(900,1500)|r|<0,-500>[E_{3}`E_{2};q_{3}]
 \morphism(900,1500)<1100,0>[E_{3}`K_{3n-2};k_{1}^{4}]
 \morphism(900,2000)|r|<0,-500>[E_{4}`E_{3};q_{4}]
 \morphism/-->/<900,500>[G_{3,n}`E_{1};g]
 \morphism/-->/<900,1000>[G_{3,n}`E_{2};h]
 \morphism/-->/<900,1500>[G_{3,n}`E_{3};l]
 \efig\]
The $k$-invariants are produced from the following relations.
\begin{table}[ht]
\label{eqtable}
\renewcommand\arraystretch{1.5}
\noindent\[
\begin{array}{|l|}
\hline
k_{1}^{2}: \quad (Sq^{2}+w_{1}^{2}+w_{2})Sq^{1}w_{3n-6}=0\\
\hline
k_{2}^{2}: \quad [(Sq^{4}+w_{2}^{2}+w_{4})Sq^{1}+(w_{1}w_{2}+w_{3})Sq^{2}+(w_{1}^{2}+w_{2})Sq^{3}]w_{3n-6}=0\\
\hline
k_{3}^{2}: \quad [(Sq^{4}+w_{2}^{2}+w_{4})Sq^{2}+(w_{1}w_{2}+w_{3})Sq^{3}]w_{3n-6}=0\\
\hline
k_{1}^{3}: \quad (Sq^{2}+w_{1}^{2}+w_{2})k_{1}^{2}=0\\
\hline
k_{2}^{3}: \quad (Sq^{2}Sq^{1}+w_{1}w_{2}+w_{3})k_{1}^{2}+Sq^{1}k_{2}^{2}=0\\
\hline
k_{1}^{4}: \quad (Sq^{2}+w_{1}^{2}+w_{2})k_{1}^{3}+Sq^{1}k_{2}^{3}=0\\
\hline
\end{array}
\]
\end{table}

We start by applying Lemma \ref{l27} (a): $f_{\nu}^{*}(w_{3n-6})=w_{3n-6}(\nu )=0$. Hence, there is a lifting $g_{1}:G_{3,n}\rightarrow E_{1}$ of $f_{\nu}$.

\medskip

Reasoning as before, if we take a class $\alpha \in H^{3n-7}(G_{3,n};\mathbb{Z}_{2})$, we get another lifting $g:G_{3,n}\rightarrow E_{1}$ such that the following relations hold (we use the notation from lemmas \ref{l28}-\ref{l31}):
\[g^{*}(k_{1}^{2})=g_{1}^{*}(k_{1}^{2})+F_{1}(\alpha ); \quad g^{*}(k_{2}^{2})=g_{1}^{*}(k_{2}^{2})+F_{2}(\alpha ); \quad g^{*}(k_{3}^{2})=g_{1}^{*}(k_{3}^{2})+F_{3}(\alpha ).\]

The $k$-invariant $k_{1}^{3}$ is produced by the relation $(Sq^{2}+w_{1}^{2}+w_{2})k_{1}^{2}=0$ which holds in $H^{*}(E_{1};\mathbb{Z}_{2})$. Pulling this relation back to $H^{*}(G_{3,n};\mathbb{Z}_{2})$ by $g_{1}^{*}$, we see that $g_{1}^{*}(k_{1}^{2})$ is in the kernel of the map $D$ (from Lemma \ref{l29}). According to Corollary \ref{c2}, the classes $w_{1}^{2}w_{3}^{n-2},w_{1}w_{2}^{2}w_{3}^{n-3},w_{2}^{4}w_{3}^{n-4}$ and $w_{2}w_{3}^{n-2}$ form a vector space basis for $H^{3n-4}(G_{3,n};\mathbb{Z}_{2})$ and by looking at lemmas \ref{l28} and \ref{l29}, one easily verifies that $\mathrm{ker}D\subseteq \mathrm{im}F_{1}$. This means that $g_{1}^{*}(k_{1}^{2})$ is in the image of the map $F_{1}$, so we can choose a class $\alpha$ such that $g^{*}(k_{1}^{2})=0$.

The group $H^{3n-2}(G_{3,n};\mathbb{Z}_{2})$ is (additively) generated by the classes $w_{1}w_{3}^{n-1}$ and $w_{2}^{2}w_{3}^{n-2}$ (Corollary \ref{c2}). By Lemma \ref{l30}, we can modify the class $\alpha$ by adding a class of the form $\alpha'=a(w_{1}^{3}w_{2}w_{3}^{n-4}+w_{1}^{2}w_{3}^{n-3})+b(w_{2}^{7}w_{3}^{n-7}+w_{2}w_{3}^{n-3})$, $a,b\in \{ 0,1\}$, and achieve the equality $g^{*}(k_{2}^{2})=0$. According to Lemma \ref{l28}, $\alpha'\in \mathrm{ker}F_{1}$, so the relation $g^{*}(k_{1}^{2})=0$ still holds.

If $g^{*}(k_{3}^{2})\neq 0$ in $H^{3n-1}(G_{3,n};\mathbb{Z}_{2})\cong \mathbb{Z}_{2}$ for this choice of $\alpha$, we modify $\alpha$ by adding the class $w_{1}w_{2}^{5}w_{3}^{n-6}+w_{2}^{4}w_{3}^{n-5}$. Lemmas \ref{l28}, \ref{l30} and \ref{l31} ensure that now we have $g^{*}(k_{1}^{2})=g^{*}(k_{2}^{2})=g^{*}(k_{3}^{2})=0$. Therefore, there is a lifting $h_{1}:G_{3,n}\rightarrow E_{2}$ of $f_{\nu}$.

\medskip

Again, by taking classes $\beta \in H^{3n-5}(G_{3,n};\mathbb{Z}_{2})$ and $\gamma \in H^{3n-3}(G_{3,n};\mathbb{Z}_{2})$, we obtain another lifting $h:G_{3,n}\rightarrow E_{2}$ and the following equalities hold (we use the notation from Lemma \ref{l32}):
\[h^{*}(k_{1}^{3})=h_{1}^{*}(k_{1}^{3})+H(\beta );\]
\[h^{*}(k_{2}^{3})=h_{1}^{*}(k_{1}^{3})+(Sq^{2}Sq^{1}+w_{1}(\nu )w_{2}(\nu )+w_{3}(\nu ))(\beta )+Sq^{1}\gamma .\]
Using the fact that the classes $w_{1}w_{2}w_{3}^{n-2},w_{2}^{3}w_{3}^{n-3}$ and $w_{3}^{n-1}$ form a vector space basis for $H^{3n-3}(G_{3,n};\mathbb{Z}_{2})$ (Corollary \ref{c2}), from Lemma \ref{l32} it is obvious that we can find a class $\beta$ such that $h_{1}^{*}(k_{1}^{3})=H(\beta )$. Also, according to Lemma \ref{l33} and Corollary \ref{c2}, by choosing appropriate class $\gamma$ (without changing $\beta$), one can obtain a lifting $h:G_{3,n}\rightarrow E_{2}$ with the property $h^{*}(k_{1}^{3})=h^{*}(k_{2}^{3})=0$, i.e., a lifting $h$ which lifts up to $E_{3}$.

\medskip

Finally, by looking at the relation that produces the $k$-invariant $k_{1}^{4}$ and according to Lemma \ref{l33} and Corollary \ref{c2} again, one observes that the indeterminacy of $k_{1}^{4}$ is all of $H^{3n-2}(G_{3,n};\mathbb{Z}_{2})$, so there is a lifting $l:G_{3,n}\rightarrow E_{3}$ of $f_{\nu}$ which lifts up to $E_{4}$. This concludes the proof of the theorem.     \hfill $\Box$

\bibliographystyle{amsplain}

\end{document}